\documentclass{amsart}
\usepackage[T1]{fontenc}
\usepackage{biblatex}
\RequirePackage{fix-cm}

\usepackage{amscd}

\newtheorem{theorem}{Theorem}
\newtheorem{proposition}[theorem]{Proposition}%
\newtheorem{corollary}[theorem]{Corollary}
\newtheorem{lemma}[theorem]{Lemma}
\newtheorem{definition}{Definition}[section]

\newtheorem{remark}{Remark}[section]

\usepackage[pdfpagelabels=true]{hyperref}

\newcommand{\disp}{\displaystyle}

\newcommand{\Ric}{\mathrm{Ric}}

\allowdisplaybreaks

\title[MCF of graphs with asymptotic Dirichlet conditions]{Mean curvature flow of graphs with asymptotic Dirichlet conditions in Cartan-Hadamard manifolds}

\author[C.R.L.~Fernandes, J.H.S.~de Lira and M.N.~Soares]{Claudia Fernandes, Jorge de Lira and Matheus Soares}

\begin{document}

\begin{abstract}
A priori estimates for the mean curvature evolution of Killing graphs in Cartan–Hadamard manifolds with asymptotic Dirichlet conditions are established. As an application, the existence of the corresponding parabolic flow is proved, ensuring regularity of the obtained solutions through the construction of suitable barriers at points of the asymptotic boundary. Such a construction is made possible under an appropriate notion of convexity at infinity.
\end{abstract}
\maketitle

\section{Introduction}
Let $P$ be a Cartan-Hadamard manifold. Given a positive function $\varrho \in C^\infty(P)$, we consider the product manifold ${M = P \times_{\varrho} \mathbb{R}}$ endowed with warped metric
\[
\bar{g} = \varrho^2(x)  \,{\rm d}s^2 + g,
\]
where $s$ is the natural coordinate in $\mathbb{R}$ and $g$ is the induced Riemannian metric in each  totally geodesic leaf $P\times \{s\}$, $s\in \mathbb{R}$. The coordinate vector field $X = \partial_s$ is a Killing vector field whose norm  $|X|=\varrho$  is preserved along the flow lines. 

The Killing graph of a function $u\in C^2(P)$ is by definition the hypersurface in $M$ given by
\begin{equation}
\Sigma[u] = \{\Phi(x,u(x)) :  x\in P\},
\end{equation}
{where $\Phi$ is the flow map of the Killing vector field.} A one parameter family of functions $u: P \times [0,T)\to \mathbb{R}$, $T>0$, defines a mean curvature flow of Killing graphs
\begin{equation}
\label{mcf1}
\Psi(x, t) = \Phi(x, u(x, t))
\end{equation}
if and only if 
\begin{align}
\label{mod-flow}
\partial_t \Psi= n {\bf H},
\end{align}
where $\textbf{H} = HN$ is the  mean curvature vector of the Killing graph $\Sigma_t :=\Sigma[u(\cdot,\,t)]$. Here,  $H$ is the scalar  mean curvature of  $\Sigma_t $ calculated with respect to the orientation given by the unit normal vector field 
\begin{equation}
N = N|_{\Psi(\cdot, \, t)} = \frac{1}{W} (\varrho^{-2} X - \nabla^P u),
\end{equation} 
where $W = (\varrho^{-2}+|\nabla^P u|^2)^{\frac{1}{2}}$ and  $\nabla^P$ denotes the Riemannian connection in $(P, g)$.

The aim of this paper is to prove the existence of Killing graphs evolving by their mean curvature from an initial graph with a prescribed asymptotic boundary. In analytical terms, this means to solve an initial value problem for a nonlinear parabolic equation with a Dirichlet boundary condition given by a function defined on the asymptotic boundary $\partial_\infty P$ of $P$. It is a well-known geometric fact that $P\cup \partial_\infty P$ can be endowed with a suitable topology that permits to give a precise formulation for such asymptotic Cauchy-Dirichlet problem. 
 For instance, in \cite{Holo:20} the authors establish some conditions under which the warped product $M= P \times_{\varrho}\mathbb{R}$ is also a Cartan-Hadamard manifold. In this context, it is natural to define the Killing graph $\Gamma$ of the boundary asymptotic data  $\varphi$ as the asymptotic boundary $\partial_{\infty} \Sigma_t$ of each graph $\Sigma_t$ in the evolving family defined by \eqref{mcf1} and \eqref{mod-flow}. Our results extend to this general geometric setting earlier  contributions by  Lin and Xhiao \cite{Lin:2010} about the existence and uniqueness of a modified mean curvature flow with prescribed asymptotic boundary of star-shaped hypersurfaces in hyperbolic space. Our main theorem (see Theorem \ref{main-2}) also extends to the parabolic setting the existence result previously obtained for stationary minimal graphs in \cite{Holo:20}.

\par Our general existence result is valid under some conditions on the topology and curvature of $P$ and $P\times_\varrho\mathbb{R}$. More explicitly, we say that $P$ satisfies the strict convexity condition (SC condition for short), if for any $x\in \partial_{\infty} P$ and a relatively open subset $W\subset \partial_{\infty} P$ contain $x$, there exists a $C^2$ open subset $U\subset P$ such that $x \in {\rm int}(\partial_{\infty} U) \subset W$ and $P-U$ is convex. There are several examples of manifolds that satisfies that properties. For example, Ripoll and Telichevesky \cite{Ripoll:15} proved that if $P$ is rotationally symmetric and satisfies $K_p \leq -\kappa^2 < 0$, then $P$ satisfies the SC condition. 

\medskip

\noindent {\textbf{Plan of the paper: }The structure of the paper is as follows. In Section \ref{Pre}, we recall some basic facts about the mean curvature flow and introduce a more manageable version of the PDE appearing in our main theorem. Moreover, we present the topological setting of our problem along with a result that guarantee smoothness in the asymptotic boundary. 
Section \ref{Est} is dedicated to present a priori estimates: Height, gradient at boundary, gradient at interior and finally curvature estimates. The paper concludes in Section \ref{Existence} with the proof of the main theorem, presented in two steps: first for the compact case, and then extended to the general ambient setting.}

\section{Preliminaries}\label{Pre}
\subsection{Geometric setting and main result}

In this section, we introduce some geometric notation and assumptions as well as a variational setting for the mean curvature flow of Killing graphs. Given a function $u$ defined in a bounded domain $\Omega \subset P$, the induced metric and volume element of the Killing graph $\Sigma[u(\cdot, t)]$ are respectively given by
\begin{equation}
g+\varrho^2 {\rm d}u \otimes {\rm d}u
\end{equation}
and
\begin{equation}
{\rm d} \Sigma_t=\varrho\sqrt{\varrho^{-2}+|\nabla^P u|^2}\, {\rm d}P.
\end{equation}
Hence the area of $\Sigma[u(\cdot, t)]$ is given by the functional 
\begin{equation*}
\mathcal{A}\left[u\right]=\int_{\Omega}\varrho\sqrt{\varrho^{-2}+|\nabla^P u|^2}\, {\rm d}P.
\end{equation*}
For an arbitrary compactly supported function $v\in C_0^\infty(\Omega)$ we have
\[
\frac{{\rm d}}{{\rm d} t}\Big|_{t=0}\mathcal{A}\left[u+t v\right]=-\int_\Omega\Big(\operatorname{div}_P\Big(\frac{\nabla^P u}{W}\Big)+\Big\langle\nabla^P\log\varrho,\frac{\nabla^P u}{W}\Big\rangle \Big)v\varrho\, {\rm d}P,
\]
where the differential operators $\nabla^P$ and $\operatorname{div}_P$ are taken with respect to the metric $g$ in $P$. Then the Euler-Lagrange equation of the functional $\mathcal{A}$ is
\begin{equation}
\label{PDE}
nH = \operatorname{div}_P\bigg(\frac{\nabla^P u}{W}\bigg)+\bigg\langle\nabla^P\log\varrho,\frac{\nabla^P u}{W}\bigg\rangle ,
\end{equation}
where $H$ is the mean curvature of the Killing graph of $u$. Once differentiating (\ref{mcf1}) with respect to $t$ we have
\[
\partial_t \Psi = \partial_t u \, X,
\]
we conclude that (\ref{mod-flow}) is equivalent to 
\begin{eqnarray*}
	\partial_t u \, X = \bigg( \operatorname{div}_P\bigg(\frac{\nabla^P u}{W}\bigg)+\bigg\langle\nabla^P\log\varrho,\frac{\nabla^P u}{W}\bigg\rangle\bigg) N.
\end{eqnarray*}
Taking the normal projection on both sides yields
\[
\partial_t u \langle X, N\rangle  = \operatorname{div}_P\bigg(\frac{\nabla^P u}{W}\bigg)+\bigg\langle\nabla^P \log\varrho,\frac{\nabla^P u}{W}\bigg\rangle.
\]
Since $\langle X, N\rangle = 1/W$ we conclude that (\ref{mcf1}) defines a mean curvature flow if and only
if $u(\cdot, \, t)$ satisfies the parabolic equation
\begin{equation}
\label{pde-q}
\partial_t u =  \mathcal{Q}[u],
\end{equation}
where
\begin{equation}
\label{opQ}
\mathcal{Q}[u]=W\bigg(\operatorname{div}_P\bigg(\frac{\nabla^P u}{W}\bigg)+\bigg\langle\nabla^P\log\varrho,\frac{\nabla^P u}{W}\bigg\rangle\bigg).
\end{equation}
In general, this non-parametric formulation is equivalent to the mean curvature flow (\ref{mod-flow}) up to tangential diffeomorphisms of the evolving graphs $\Sigma_t$. The equivalence here follows from the fact that we are assuming a fixed gauge, namely the choice of coordinates fixed in (\ref{mcf1}). 
\\
In order to control the geometry of our problem, we assume that that there exists constants $L, L_1 > 0$ such that
\begin{align}
\label{ricci-cond}
\operatorname{Ric}_{\bar g} \ge - L_1\overline{g} \quad \mbox{and}  \quad \operatorname{Ric}_g - \nabla^2 \log \varrho \ge -Lg
\end{align}
Moreover, we need to impose further curvature conditions for the validity of comparison theorem that will be useful for the  construction of geometric barriers. Let  $(r, \vartheta)\in\mathbb{R}^+\times\mathbb{S}^{n-1}$ be Gaussian global coordinates in $P$ defined with respect to a fixed point $o\in P$.
 We suppose that the radial sectional curvatures along geodesics rays issuing from $o$ satisfies
\begin{equation}
\label{K-comp}
-\frac{\iota''(r)}{\iota(r)}\geq K(\partial_r \wedge {\sf v}) \ge - \frac{\xi''(r)}{\xi(r)}
\end{equation}
for all $r>0$,  ${\sf v} \in TM, {\sf v}\perp \partial_r$, where  $\xi, \iota \in C^\infty([0, \infty))$ is a function  satisfying the following conditions
\begin{equation}
\begin{split}
\label{xi-model}
& \xi(r)> 0,\, \, \iota(r)>0 \,\, \mbox{ for }\,\ r > 0,  \\
& \xi'(0) = \iota'(0) = 1,\\
& \xi^{(2k)}(0) = 0, \,\, \mbox{ for }\,\, k \in \mathbb{N} \cdot
\end{split}
\end{equation}
In this case, the Hessian comparison theorem \cite{AMR} implies that
\begin{equation}
\label{hess-comp}
\frac{\iota'(r)}{\iota(r)} \left(g-{\rm d}r\otimes {\rm d}r\right) \le
\nabla^P \nabla^P r \le \frac{\xi'(r)}{\xi(r)} \left(g-{\rm d}r\otimes {\rm d}r\right).
\end{equation}
We also suppose that the warping function $\varrho$ is rotationally invariant, that is, 
\begin{equation}
\label{rho-model-0}
\varrho (x) =  \varrho(r(x)), \, \, \mbox{for} \, \, x \in P.
\end{equation}
where $r(x) = \operatorname{dist}(o,x)$  and
\begin{align}
\label{rho-model}
& \varrho(r)> 0, \, \, \varrho'(r) > 0  \,\, \mbox{ for }\,\, r>0,  \\
& \varrho(0) = 1,\,\, \varrho^{(2k+1)}(0) =0, \,\, \mbox{ for }\,\, k \in \mathbb{N}, \\
& \liminf_{r \to \infty} \frac{\varrho'(r)}{\varrho(r)} > 0.
\end{align}
Moreover, we suppose that 
\begin{equation}
\label{cylinder-0}
\begin{split}
\frac{\varrho'(r)}{\varrho(r)} \le \frac{\xi'(r)}{\xi(r)} \quad \mbox{ and } \quad
\frac{\varrho'(r)}{\varrho(r)} \le \frac{\iota'(r)}{\iota(r)}
\end{split}
\end{equation}

For the sake of brevity, We say that a manifold $P^n$ satisfies the \textit{geometric conditions} if all properties from \eqref{ricci-cond} to \eqref{cylinder-0} are valid. 
\begin{theorem}\label{main-2}

Let $P$ be a Cartan-Hadamard manifold and $\varrho\in C^\infty(P)$. Suppose that $P\times_\varrho \mathbb{R}$ satisfies the geometric conditions \eqref{ricci-cond} to \eqref{cylinder-0}. Given an entire smooth Killing graph $\Sigma_0$ with asymptotic boundary $\Gamma$, there exists a mean curvature flow of entire Killing graphs $\Sigma_t, t\ge 0$,  with initial condition $\Sigma_0$ and asymptotic boundary $\Gamma$.
\end{theorem}

In analytical and more precise terms, the theorem above may be rewritten in terms of the existence of  
	 a solution $u\in C^{\infty}(P\times [0,\infty))\cap C(\bar{P}\times [0, \infty) )$ of the Cauchy-Dirichlet problem 	
	\begin{equation}\label{main-problema}
	\begin{cases}
	&\partial_t u= \operatorname{div}_P\Big(\frac{\nabla^P u}{W}\Big)+\Big\langle\nabla^P \log\varrho,\frac{\nabla^P u}{W}\Big\rangle
	 \,\,  {\rm in} \,\, P \times [0,\infty), \\
	& u(x,0)  = u_0(x)  \,\, {\rm for} \,\, x\in P \times \{0\}, \\
	& u(x,t) = \varphi (x)  \,\, {\rm for}\,\, (x,t) \in \partial_{\infty} P \times [0,\infty),
	\end{cases}
	\end{equation}
under the assumptions that $P \times_\varrho \mathbb{R}$ satisfies the geometric conditions and that $P$ is  strictly convex. In the statement of the theorem,  
     $\Sigma_0 $ is the Killing graph of the initial data $u_0 \in C^{\infty}(P) \cap C(\bar{P})$. Moreover, the asymptotic boundary of $\Sigma_0$ (and of the Killing graphs $\Sigma_t$ of $u(\cdot, t)$, for $t>0$) is the graph of $u_0|_{\partial_\infty P}$ given by the submanifold  $\partial_{\infty} \Sigma_0 = \Gamma$. 

\subsection{Analytical inequalities}
    
Now we deduce some evolution equations that will be useful in the sequel.
\begin{proposition}
	\label{prop-par-s} Suppose that {\rm(}\ref{K-comp}{\rm)} holds. The restrictions of the functions $r$ and $s$ to the graphs $\Sigma_t$, $t\in [0, T)$, satisfy
	\begin{equation}
	\label{par-r}
	\begin{split}
	& (\partial_t -\Delta) r \ge -\frac{\xi'(r)}{\xi(r)} \left(n- |\nabla r|^2\right)- \varrho^2|\nabla s|^2 \left(\langle\bar\nabla\log\varrho, \nabla r\rangle -\frac{\xi'(r)}{\xi(r)}\right)
	\end{split}
	\end{equation}
	and
	\begin{equation}
	\label{par-s}
	(\partial_t - \Delta )s = -2\langle\bar{\nabla}\log\varrho,N\rangle\langle\bar{\nabla}s,N\rangle.
	\end{equation}
	In both expressions, $\nabla$ and $\Delta$ are the intrinsic Riemannian connection and Laplacian in $\Sigma_t$, respectively, whereas $\bar\nabla$ denotes the Riemannian connection in $M$. Moreover, given the function
	\begin{equation}
	\label{defn-zeta}
	\zeta(\Psi(x, t)) = \int_0^{r(\Psi(t,x))} \xi(\varsigma)\, {\rm d}\varsigma
	\end{equation}
	we have
	\begin{equation}
	\label{par-zeta}
	\begin{split}
	& (\partial_t -\Delta) \zeta \ge -n\xi'(r) - \varrho^2|\nabla s|^2 \xi(r) \left(\langle\bar\nabla\log\varrho, \nabla r\rangle -\frac{\xi'(r)}{\xi(r)}\right).
	\end{split}
	\end{equation}
\end{proposition}

\begin{proof} Observe that  $\bar \nabla s = \varrho^{-2}X$ and $\nabla s = \varrho^{-2}X^\top$, where $\top$ denotes the tangential projection onto $T\Sigma_t$. Given a local orthonormal tangent frame $\{{\sf e}_i\}_{i=1}^n$ in $\Sigma_t$, one has
	\begin{equation*}
	\begin{split}
	\Delta s 
	&=\langle\nabla\varrho^{-2},X^\top\rangle+\varrho^{-2}\sum_{i=1}^n\langle\bar{\nabla}_{{\sf e}_i}X,{\sf e}_i\rangle+nH\langle\varrho^{-2} X,N\rangle
    \\
    &=\langle\bar{\nabla}\varrho^{-2},X^\top\rangle+nH\langle\bar{\nabla}s,N\rangle\\
	&= -\langle\bar{\nabla}\varrho^{-2},N\rangle\langle X, N\rangle+nH\langle\bar{\nabla}s,N\rangle \\
    &= 2\langle\bar{\nabla}\log\varrho,N\rangle\langle \bar\nabla s, N\rangle+nH\langle\bar{\nabla}s,N\rangle.
	\end{split}
	\end{equation*}
	We also compute
	\[
	\partial_t s=\langle\bar{\nabla}s, \partial_t \Psi\rangle=nH\langle\bar{\nabla}s,N\rangle.
	\]
	Therefore
	\[
	(\partial_t -\Delta)s =  -2\langle\bar{\nabla}\log\varrho,N\rangle\langle\bar{\nabla}s,N\rangle.
	\]
	Now we obtain
	\[
	\langle \bar\nabla_X \bar\nabla r, X\rangle = \langle \bar\nabla_{\bar\nabla r} X, X\rangle = \frac{1}{2}\partial_r |X|^2 = \frac{1}{2}\partial_r \varrho^2  =
	\varrho \langle\bar\nabla\varrho, \bar\nabla r\rangle.
	\]
	Fixed a local orthonormal tangent frame $\{{\sf e}_i\}_{i=1}^n$ in $\Sigma_t$, we have
	\begin{equation}
	\begin{split}
	& \Delta r = \sum_i\langle\nabla_{{\sf e}_i}\nabla r,{\sf e}_i\rangle= \sum_i\langle\bar\nabla_{{\sf e}_i}(\bar\nabla r-\langle\bar\nabla r, N\rangle N),{\sf e}_i\rangle\\
	&=\sum_i \langle\nabla^P_{\pi_*{\sf e}_i}\pi_*\bar\nabla r,\pi_* {\sf e}_i\rangle+\frac{1}{\varrho^4} \sum_{i}\langle {\sf e}_i, X\rangle^2\langle\bar\nabla_{X}\bar\nabla r, X\rangle-\sum_i \langle\bar\nabla r,N\rangle\langle\bar{\nabla}_{{\sf e}_i}N,{\sf e}_i\rangle\\
	& =  \sum_i \langle\nabla^P_{\pi_*{\sf e}_i}\nabla^P r,\pi_* {\sf e}_i\rangle+ |\nabla s|^2 \langle\varrho\bar\nabla\varrho, \nabla r\rangle+nH\langle\bar\nabla r,N\rangle\nonumber
	\end{split}
	\end{equation}
	where $\pi: \bar M = P \times \mathbb{R} \to P$ is the projection on the first factor, that is, $\pi(s,x) =x$ for all $(s,x)\in P\times \mathbb{R}$. The Hessian comparison theorem (\ref{hess-comp}) implies that
	\begin{equation}
	\begin{split}
	& \Delta r  \le \frac{\xi'(r)}{\xi(r)} \sum_i\big(|\pi_*{\sf e}_i|^2 - \langle {\sf e}_i, \nabla^P r\rangle^2\big)+ |\nabla s|^2 \langle\varrho\bar\nabla\varrho, \nabla r\rangle+nH\langle\bar\nabla r,N\rangle
	\\
	& = \frac{\xi'(r)}{\xi(r)} \left(n- \varrho^2|\nabla s|^2 - |\nabla r|^2\right)+ \varrho^2|\nabla s|^2 \langle\bar\nabla\log\varrho, \nabla r\rangle+nH\langle\bar\nabla r,N\rangle.
	\end{split}
	\end{equation}
	Hence,
	\begin{equation}
	\label{deltar-final}
	\Delta r \le  \frac{\xi'(r)}{\xi(r)} \left(n-  |\nabla r|^2\right)+ \varrho^2|\nabla s|^2\left( \langle\bar\nabla\log\varrho, \nabla r\rangle- \frac{\xi'(r)}{\xi(r)}\right)+nH\langle\bar\nabla r,N\rangle.
	\end{equation}
	We also have 
	$\nabla \zeta = \xi(r)\nabla r$ and 
	\begin{equation}
	\begin{split}
	\Delta \zeta&= \xi(r) \Delta r +\xi'(r) |\nabla r|^2. 
	\end{split}
	\end{equation}
	Therefore
	\begin{equation}
	\begin{split}
	&  \Delta \zeta \le n\xi'(r) + \varrho^2|\nabla s|^2 \xi(r)\left( \langle\bar\nabla\log\varrho, \nabla r\rangle- \frac{\xi'(r)}{\xi(r)}\right)+nH\xi(r)\langle\bar\nabla r,N\rangle.
	\end{split}
	\end{equation}
	On the other hand
	\[
	\partial_t r = \left\langle \bar\nabla r, \partial_t\Psi\right\rangle = nH \langle \bar\nabla r, N\rangle
	\]
	and
	\[
	\partial_t \zeta = nH \xi(r)\langle \bar\nabla r, N\rangle.
	\]
	We conclude that 
	\begin{equation}
	\begin{split}
	& (\partial_t -\Delta) r \ge -\frac{\xi'(r)}{\xi(r)} \left(n- |\nabla r|^2\right)- \varrho^2|\nabla s|^2 \left(\langle\bar\nabla\log\varrho, \nabla r\rangle -\frac{\xi'(r)}{\xi(r)}\right)
	\end{split}
	\end{equation}
	and
	\begin{equation}
	\begin{split}
	& (\partial_t -\Delta) \zeta \ge -n\xi'(r) - \varrho^2|\nabla s|^2 \xi(r) \left(\langle\bar\nabla\log\varrho, \nabla r\rangle -\frac{\xi'(r)}{\xi(r)}\right).
	\end{split}
	\end{equation}
	This finishes the proof of the proposition. 
\end{proof}

\begin{proposition}\label{evolutionW} If the graphs $\Sigma_t$, $t\in [0, T]$, evolve by the mean curvature flow \eqref{mcf1}-\eqref{mod-flow}, then
	\begin{equation}\label{evolW}
	\left(\partial_t-\Delta\right)W=-W(|A|^2+\overline{\Ric}(N,N))-2W^{-1}|\nabla W|^2,
	\end{equation}
	where $W = \langle X, N\rangle^{-1} = (\varrho^{-2}+|\nabla^M u)|^2)^{1/2}$ and $A$ is the Weingarten map of $\Sigma_t$. 
\end{proposition}
\begin{proof} Note that
	\begin{equation}
	\label{nablaX}
	\nabla\langle X,N\rangle=\langle X,N\rangle (\bar\nabla\log\varrho)^\top-\langle\bar\nabla\log\varrho,N\rangle X^\top-AX^\top.
	\end{equation}
	Hence,
	\begin{equation*}
	\begin{split}
	& \langle X,N\rangle (\bar\nabla\log\varrho)^\top-\langle\bar\nabla\log\varrho,N\rangle X^\top = \langle X, N\rangle \bar\nabla\log\varrho - \langle\bar\nabla\log\varrho, N\rangle X.
	\end{split}
	\end{equation*}
	It follows from the second variation formula for the functional $\mathcal{A}_0$ that 
	\[
	\Delta\langle X,N\rangle+|A|^2\langle X,N\rangle+\overline{\Ric}(N,N)\langle X,N\rangle=-n\langle\nabla H,X^\top\rangle,
	\]
	where $|A|$ stands for the norm of the Weingarten map of $\Sigma_t$ and $\top$ denotes the tangencial projection onto $T\Sigma_t$. 
	On the other hand, using that $X$ is a Killing vector field one gets
	\begin{eqnarray*}
		\partial_t\langle X,N\rangle &=&\langle \bar{\nabla}_{\partial_t}X,N\rangle+\langle X,\bar{\nabla}_{\partial_t}N\rangle= nH\langle \bar\nabla_N X, N\rangle - n \langle X, \nabla H\rangle\\
		&=&
		-n\langle X^\top,\nabla H\rangle,
	\end{eqnarray*}
	where $\bar{\nabla}$ denotes the Riemannian connection in $\bar{M}$
	and so
	\[
	\left(\partial_t-\Delta\right)\langle X,N\rangle=|A|^2\langle X,N\rangle+\overline{\Ric}(N,N)\langle X,N\rangle
	\]
	Thus,  using that  $\langle X,N\rangle=1/W$ one has
	\[
	\partial_tW = -W^2 \partial_t W^{-1}=-W^2\partial_t\langle X,N\rangle
	\]
	and
	\begin{equation}\label{four}
	\Delta W-\frac{2}{W}|\nabla W|^2 =-W^2 \Delta W^{-1} = -W^2 \Delta\langle X, N\rangle.
	\end{equation}
	Hence, one concludes that
	\begin{eqnarray}\nonumber
	\partial_t W-\Delta W+\frac{2}{W}|\nabla W|^2
	=-W^2\left(\partial_t-\Delta\right)\langle X,N\rangle
	=-W(|A|^2+\overline{\Ric}\left(N,N)\right)
	\end{eqnarray}
\end{proof}

\subsection{Topological setting}

As cited before in the introduction, an important part of work is the geometry of the boundary. The natural way to introduce a geometry of boundary in Cartan-Hadamard manifold is through the structure of asymptotic boundary (cf \cite{Eberlein:96}, \cite{Wang:98}, \cite{Ripoll:15}). This notion allows to extend geometric and analytic considerations from the interior points of the manifold to its points at infinity, offering a comprehensive way to understand the behavior of the manifold globally. In this perspective, in this subsection, our main objective is introduce some basic facts about the cone topology and the structure of asymptotic boundary.

The asymptotic boundary $\partial_{\infty} P$ of $P$ is a smooth manifold defined as the set of equivalence classes of geodesics in $P$ with respect to the following equivalence relation. The equivalence class $\alpha(\infty)$ of some unit speed geodesic $\alpha:\mathbb{R} \to P$ consists in all unit speed geodesics $\beta: \mathbb{R} \to P$ which are asymptotic to $\alpha$ in the sense that 
\[
{\sup}_{t} \operatorname{dist}(\alpha(t), \beta(t)) < \infty.
\]
From now on we will denote $ \bar{P} = P \cup \partial_{\infty}P$.  For each $x \in P$ and $y \in \bar{P} \setminus \{x\}$ there exists a unique unit speed geodesic $\gamma^{x,y} : \mathbb{R} \to P$ such that $\gamma^{x,y}(0)=x$ and  $\gamma^{x,y}(t)=y$ for some $t \in (0, \infty].$ Given $(x, {\sf v})$ in the unit tangent bundle of $P$ and fixed $\delta >0$ and $r>0$, we define the cone
  \[
  C(x, {\sf v}, \delta) = \{y \in \bar{P} \setminus \{x\}: \angle({\sf v}, \dot{\gamma}^{x,y}(0)) < \delta \}
  \]
  and the truncated cone 
  \[
  T(x, {\sf v}, \delta, r) = C(x, {\sf v}, \delta) \setminus \bar{B}_x(r)
  \]
  where $\angle({\sf v}, \dot{\gamma}^{x,y}(0))$ is the angle between vectors ${\sf v}$ and $\dot{\gamma}^{x,y}(0)$ in $T_xP$ and $B_x(r)$ is the geodesic ball of radius $r$ centred at $x$. The set of all truncated cones and geodesic balls 
  defines a basis of topology on $\bar{P}$ which  is called cone topology. The manifold $\bar{P}$ is the compactification of $P$ according to cone topology.

 In order to study the behavior of the solutions of (\ref{main-problema})  we will use a notion of the barrier and regularity at infinity when $P$ satisfies the strict convexity condition. For this, we will proceed as Ripoll-Telichevesky in \cite{Ripoll:15}.  First we recall the notions of subsolutions (supersolutions) and lower barriers (upper barriers).

 We say that $\eta \in C^0(P \times [0,\infty))$ is a supersolution (subsolution) of $\partial_t - \mathcal{Q}$ if given any bounded domain $U \subset P \times [0,\infty)$ and $u \in C^0(\bar{U})$ such that $(\partial_t - \mathcal{Q})(u) = 0$ in $U$ with $u|_{\partial U} \leq (\geq) \, \eta|_{{\partial U}}$ then $u|_U \leq   (\geq) \,  \eta|_{U}.$
	If $v \in C^2(U)$ and $(\partial_t - \mathcal{Q})(v) \geq 0 \, (\leq 0),$ then $v$ is a supersolution (subsolution).

\begin{definition}
Given $(x_0, t_0) \in \partial_{\infty} P \times [0, \infty)$, a constant $C >0$  and open subsets $U \subset P$, $I \subset  [0, \infty)$  such that $x_0 \in \partial_{\infty}U$ and $t_0 \in I$,  a function $\eta \in C^0(P \times [0, \infty))$ is an upper barrier for $\partial_t - \mathcal{Q}$ relative to $(x_0, t_0)$ and $U \times I $ with height $C$ if 
\begin{enumerate}
\item[i.] $\eta$  is a supersolution for $ \partial_t - \mathcal{Q}$; 
\item[ii.] $\eta \geq 0$ and $\eta (x,t) \to 0 $ as $(x,t) \to (x_0, t_0)$ where the limit is with respect to the cone topology;
\item[iii.]  $\eta|_{{P \times [0, \infty) \setminus \, U \times I  }} \geq C. $
\end{enumerate}
In a similar way, we define subsolutions and lower barries. 
\end{definition}
Another important definitions are the notion of regularity at infinity and, as cited before, the strict convexity condition, which is defined, respectively, of the following way:
\begin{definition} We say that
	$P \times [0,\infty)$ is regular at infinity with respect to $\partial_t - \mathcal{Q}$ if given a point $ (x_0, t_0) \in \partial_{\infty} P \times [0, \infty)$, a constant $C>0$ and open subsets  $W \subset \partial_{\infty} P$ $I \subset [0, \infty)$  with $x_0 \in W,\, t_0 \in I$ there exist  open subsets $U \subset P, \, J \subset I$ such that $x_0 \in \operatorname{int}(\partial_{\infty} U) \subset W, \, t_0 \subset J$ and there exist upper and lower barriers $\bar{\eta}, \eta $ for $\partial_t - \mathcal{Q}$ relatives to $(x_0, t_0)$ and $U \times J$ with height $C$. Here $\operatorname{int}(\partial_{\infty}U)$ denotes the interior of $\partial_{\infty}U$ in $\partial_{\infty}P$.
\end{definition}
Finally, we present our result that guarantee the regularity at infinity under some hypothesis. As we said before, the main idea is present a parabolic version of a result from Ripoll and Telichevesky \cite[Theorem 2.5]{Ripoll:15}.
\begin{definition} We say that  $P$ satisfies the strict convexity condition {\rm(}SC condition for short{\rm)}  if for any $x \in \partial_{\infty}P$ and a relatively open subset $W \subset \partial_{\infty}P$ containing $x$,  there exists a $C^2$ open subset $U \subset \bar{P}$ such that $x \in \operatorname{int}(\partial_{\infty} U) \subset W$ and $P \setminus U$ is convex. Furthermore if  $M \setminus\{(x,s): x\in U, \, s\in\mathbb{R} \}$ is also convex, we say that  $M$ satisfies the strict convexity condition.
\end{definition}
\begin{proposition}\label{reg-at-infinity}
    Let $P$ be a Hadamard manifold satisfying the $SC$ condition. Suppose that exists a function $\iota$ as equations $\eqref{K-comp}, \eqref{xi-model}$ and $\eqref{cylinder-0}$. Then, $P\times [0,\infty)$ is regular at infinity with respect to $\partial_t-Q.$
\end{proposition}
\begin{proof}
    Given $ (x_0, t_0) \in \partial_{\infty}P \times [0, \infty)$, a constant $C>0$ and open subsets $W \subset \partial_{\infty}P$ and  $I \subset [0,\infty)$ such that $x_0 \in W$, $t_0 \in I$, we consider  a $C^2$ open subset $U \subset P$ such that $x_0 \in \operatorname{int}(\partial_{\infty}U) \subset W$ and $P \setminus U$ is convex. Then we define a distance function $d : U \to \mathbb{R}$ be a distance function to $\partial U$ in $U$. Note that $y = (x,s) \in K_U = \{(x, s): x\in U, s\in\mathbb{R}\}$ one has
    \[
    \operatorname{dist}((x,s), K_{\partial U}) =  \operatorname{dist}(x, \partial U) \doteq d(x).
    \]
    Hence we also denote by $d$  the distance function to the Killing cylinder over $\partial U$ in $K_U$. Recall that $r(x)$ denotes  the radial distance $\operatorname{dist}(x,o)$ in $P$. In the same way, we also extend $r$ to $M$ as  $r(x,s) = r(x)$.
    In order to construct a upper barrier for $\partial_t - \mathcal{Q}$ relative to  $(x_0, t_0)$ and $U \times I$  with height $C$, we consider a function 
    \[
    w (x, t) = f(d(\Psi(x,t))) = f(d(x))
    \]
     in $U \times I$  where $f(d)= C_1 \exp(-\alpha d)$ where $C_1$ and $\alpha$ are positive constants to be fixed later.  
    Indicating derivatives with respect to $d$ by $\cdot$ one has $\partial_t w = \dot f (d)\partial_t d =0$. Hence we have to choose constants $C_1$ and $\alpha$ such that 
    \[
    \mathcal{Q}[w] = \Delta w- \frac{1}{W^2} \big \langle \nabla_{\nabla w} \nabla w , \nabla w \big\rangle +  \bigg(1 + \frac{1}{\varrho^2 W^2}\bigg) \big\langle \nabla \log \varrho, \nabla w \big\rangle \leq 0.
    \]
     it follows from $\nabla w = \dot f (d)\nabla d$,  $W^2 = \varrho^{-2} + \dot f^2(d)$ and $\Delta w = \dot f (d)\Delta d + \ddot f (d)$
     that
     \[
      \langle \nabla \log \varrho , \nabla w  \rangle = \dot f  (d)\frac{\varrho'(r)}{\varrho(r)} \langle \nabla r, \nabla d \rangle 
     \]
    where $\dot f$ is a derivative of $f$ with respect of $d$. We also have 
      \[
    \big \langle \nabla_{\nabla w} \nabla w, \nabla w \big \rangle  = \dot f^2  \langle \nabla_{\nabla d} \dot f \nabla d, \nabla d \rangle    
    = \dot f^2 \ddot f
    \]
    where we used that $\nabla_{\nabla d}\nabla d = 0$.  Hence
    \begin{align*}
    \mathcal{Q}[w] &=  \dot f \Delta d  +  \frac{1}{\varrho^2 W^2}\ddot f + \bigg(1 + \frac{1}{\varrho^2 W^2}\bigg)\frac{\varrho'}{\varrho} \langle \nabla r, \nabla d \rangle \dot f
    \end{align*}
    On the other hand, the fourth condition of the function $\iota$ can be rewrite as
    \begin{equation}
        \iota''-K_{rad}(P)\iota\leq 0.
    \end{equation}
    Hence, by the Laplacian comparison theorem, we have that
    \begin{equation}
        \Delta r \geq (n-1)\frac{\iota'(r)}{\iota(r)}.
    \end{equation}
    Moreover, since $P \setminus U$ is convex, we have $\langle \nabla r, \nabla d \rangle > 0$
    Thus,
    \begin{align*}
     	 & \mathcal{Q}[w] \leq  \frac{1}{\varrho^2 W^2}\ddot f + (n-1)\frac{\iota'}{\iota}\dot f + \bigg(1 + \frac{1}{\varrho^2 W^2}\bigg)\frac{\varrho'}{\varrho} \langle \nabla r, \nabla d \rangle \dot f \\
     	&\leq  \ddot f+ \bigg((n-1)\frac{\iota'}{\iota} + \frac{\varrho'}{\varrho} \langle \nabla r, \nabla d \rangle \bigg) \dot f =  \bigg((n-1)  \frac{\iota'}{\iota} + \frac{\varrho'}{\varrho} \langle \nabla r, \nabla d \rangle - \alpha \bigg) \dot f
    \end{align*}
    in $U \times I.$
    Using that $\liminf_{r \to \infty} \frac{\varrho'(r)}{\varrho(r)} > 0,$ we take $d_0 \geq 2$ and 
    \begin{equation}
        U_0 = \{ x \in U ; d(x) \geq d_0 \}
    \end{equation}
    such that $\inf_{U_0} \frac{\varrho'(r)}{\varrho(r)} \langle \nabla r, \nabla d \rangle > 0.$ Let also consider $$U_1 = \{x \in U ; d(x) > d_0 -1 \} \qquad \mbox{and} \qquad U_2 = U_1 \setminus U_0.$$
    By inequality \eqref{cylinder-0}, we have that ${\rm inf}_{U_0} \frac{\iota'(r)}{\iota(r)}>0$.
    If we choose
    \[
    0 < \alpha \leq {\inf}_{U_\ell} \frac{\varrho'(r)}{\varrho(r)} \langle \nabla r, \nabla d \rangle \quad \mbox{and} \quad C_1 = C e^{\alpha d_0}
    \]
    we have $\mathcal{Q}[w] \leq 0$ in $U_0$ and 
    \[
    {\inf}_{U_2 \times I} w(x,t) = {\inf}_{U_2} C e^{\alpha d_0} e^{- \alpha d(x)} = C = f(d_0).
    \]
    Setting
    $$\widetilde{w}(x,t) = \begin{cases}
    w(x,t) &\mbox{if} \quad x \in U_0, \, t \in I \\
    C, &\mbox{if} \quad x \in U_2,\, t \in I
    \end{cases}$$
    one has a continuous function $\widetilde{w}$ in the open subset $U_1 \times I$ that can be extended to $P \times [0, \infty)$ as
    $$\eta (x,t) = \begin{cases}
    \omega(x,t) &\mbox{if} \qquad x \in U_0, \quad  t \in [0, \infty) \\
    C, &\mbox{if} \qquad x \in P \setminus U_0 \quad t \in [0, \infty).
    \end{cases}$$
    which is an upper barrier for $\partial_t - Q$ relative to $(x_0, t_0)$ and $U \times I$ with height $C.$ In a similar way, we obtain a lower barrier relative to $(x_0, t_0)$ and $U \times I$ with height $C.$ Hence, $P \times [0, \infty)$ is regular at infinity with respect to $\partial_t - Q.$
\end{proof}
\begin{remark}
    It is important to highlight that, since we are working with the Laplacain Comparasion Theorem, the usual assumption of the radial sectional curvature $K_{rad}(P)$ being bounded can be replaced by an estimate on the Ricci curvature in the radial direction, namely 
    \begin{equation} 
    {\rm Ric}_p(\nabla r, \nabla r) \leq -(n-1)\frac{\iota''}{\iota}.
    \end{equation}
    This formulation extends the classical hypothesis, allowing a more general curvature control that still ensures the validity of the comparison results.
\end{remark}
As consequence of the Proposition \ref{reg-at-infinity}, we have
\begin{corollary}
  Let $P$ be a Hadamard manifold with sectional curvature $K_P \leq - \kappa^2$ satisfying the  SC condition     and suppose that  $\varrho$ satisfies $(\ref{rho-model}).$ Then $P \times [0, \infty)$ is regular at infinity with respect to $\partial_t - Q.$
\end{corollary}

\section{A priori estimates}\label{Est}
In this section, we present a priori estimates in order to show the existence and smoothness of the solution. It is important to highlight that the first equation in $\ref{main-2}$ can be written in local coordinates $\{x^i\}_{i=1}^n$ in $P$ as 
\begin{equation} \label{R-approximate-problem}
\partial_t u = \bigg (g^{ij}- \frac{u^iu^j}{W^2}\bigg)u_{i;j} + \left(1 + \frac{1}{\varrho^2 W^2}\right)(\log \varrho)^iu_i  
\end{equation}
where $g_{ij}$ are the local components of the metric $g$ in $P$ and $u^i$ and $u_{i;j}$ are the components of the gradient and Hessian of $u(\cdot, t)$, $t\in [0,T)$, respectively. Let $\mathcal{C}_R$ be the parabolic cylinder  $B_R(o) \times [0, T_R)$ where $B_R(o)$ is the geodesic ball in $P$ centred at $o\in P$  with radius $R>0$ and  $T_R$ is a positive time possibly dependent on $R$. Observe that this equation is uniformly parabolic in $\mathcal{C}_R$  once we have $C^1$ bounds for $u|_{\mathcal{C}_R}$.

Given $o$ we can extend the continuous asymptotic boundary data $\varphi\in C^0(\partial_{\infty} P)$ as constant along the geodesic rays issuing from $o$. This is consistent with the definition of points os the asymptotic boundary as equivalence classes. Indeed if $\gamma$ is the geodesic ray connecting $o\in P$ and $x\in \partial_\infty P$ and $\beta \sim \gamma$ then
\[
\varphi(\beta(t)) = \varphi(\beta(\infty)) = \varphi(\alpha(\infty)).
\]
In the next sections, we obtain height, gradient and Hessian bounds in $\mathcal{C}_R$. We also prove an existence result for the Cauchy-Dirichlet problem for the  mean curvature flow of graphs in $B_R(o)$ as a direct consequence of the estimates and the standard Schauder theory for quasilinear parabolic PDEs as exposed in \cite{Lieberman}. 

\subsection{Height estimate}
In order to obtain a height estimate for a solution of the problem  $(\ref{R-approximate-problem})$ we will use barriers. For this let us consider $P_+$ be a complete, non-compact, $n$-dimensional model manifold with respect to a fixed pole $o_+\in P_+$ in the sense that the Riemannian metric in $P_+$ can be expressed in Gaussian coordinates $(r, \vartheta)\in \mathbb{R}\times \mathbb{S}^{n-1}$ centered at $o_+$ as
\begin{equation}
\label{model-metric}
g_+ = {\rm d}r^2 + \xi^2(r)\, {\rm d}\vartheta^2
\end{equation} 
where ${\rm d}\vartheta^2$ denotes the round metric in $\mathbb{S}^{n-1}$ and $\xi \in C^\infty ([0, \infty))$ is the function defined in (\ref{xi-model}). 
We then consider  the warped metric in $P_+ \times \mathbb{R}$ 
\begin{equation}
\varrho^2(r) {\rm d}s^2 + {\rm d}r^2 + \xi^2(r)\, {\rm d}\vartheta^2.
\end{equation}
and we denote
\begin{equation}
A(r) =   \varrho(\varsigma)\xi^{n-1}(\varsigma)
\end{equation}
and
\begin{equation}
V(r) = \int_0^r \varrho(\varsigma)\xi^{n-1}(\varsigma)\, {\rm d}\varsigma.
\end{equation}
We also define
\begin{equation}
\label{meanHR}
H(r) = -\frac{1}{n}\frac{A(r)}{V(r)}\cdot
\end{equation}
Given $x\in P_+$ we denote $r(x) = {\rm dist}(o_+,x)$. Let $B_R(o_+)$ be the closed geodesic ball centered at $o_+$ with radius $R\in (0,\infty)$. Hence $x\in B_R(o_+)$ if and only if $r(x) \le R$. The mean curvature of the Killing cylinder over the geodesic sphere $\partial B_r(o_+)$ is given by
\begin{equation}
H_{\rm cyl}(r) = \frac{1}{n} \left((n-1)\frac{\xi'(r)}{\xi(r)}+\frac{\varrho'(r)}{\varrho(r)}\right).
\end{equation}
Fixed this notation, we are able to state the following result.

\begin{proposition}\label{graph-v_R}
	For each $R\in (0,\infty)$, the graph of the function 
	\begin{equation}
	\label{vR}
	v_R(x)  =\int^{r(x)}_R \frac{nH(R)V(\varsigma)}{\varrho(\varsigma) (A^2(\varsigma)-n^2 H^2(R)V^2(\varsigma))^{\frac{1}{2}}}\,{\rm d}\varsigma
	\end{equation}
	defined in $B_R(o_+)$ has constant mean curvature $H(R)$ and its boundary is the geodesic sphere $\partial B_R(o_+)$.
\end{proposition}

\noindent \emph{Proof.} Fixed $R\in (0,\infty)$ let $v_R$ be the (radial) solution of the following Dirichlet problem for the constant mean curvature equation
\begin{align}
\begin{cases}
& \label{PDEM} \operatorname{div}_+\left(\frac{\nabla^+ v_R}{W_+}\right)+ g_+\left(\nabla^+\log\varrho,\frac{\nabla^+ v_R}{W_+}\right)=nH(R) \,\,\mbox{ in } \,\, B_R(o_+),\\
& v_R|_{\partial B_R(o_+)} = 0,
\end{cases}
\end{align}
where the differential operators  ${\rm div}_+$ and $\nabla^+$ are defined with respect to the metric (\ref{model-metric}) in $P_+$ and
\[
W_+ = (\varrho^{-2}(r)+ v'^2_R(r))^{\frac{1}{2}},
\]
with $'$ denoting derivatives  with respect to $r$.  Note that (\ref{PDEM}) can be written in terms of a weighted divergence as
\begin{equation}
\operatorname{div}_{-\log\varrho} \left(\frac{\nabla^+ v_R}{W_+}\right) \doteq \frac{1}{\varrho} \operatorname{div}_+ \left(\varrho \frac{\nabla^+ v_R}{W_+}\right) = nH(R).
\end{equation}
Integrating with respect to the density $\varrho \, {\rm d}P_+$ yields
\begin{equation}
\begin{split}
\int_{B_r(o)} nH(R) \, \varrho \, {\rm d}P_+ &= \int_{B_r(o)} \operatorname{div}_+ \left(\varrho \frac{\nabla^+ v_R}{W}\right)\, {\rm d}P_+ \\
& = \int_{\partial B_r(o)} g_+\left( \frac{\nabla^+ v_R}{W_+}, \partial_r\right) \varrho \, {\rm d}\partial B(r),
\end{split}
\label{div-int}
\end{equation}
for $r\le R$. Since  $v_R$ is radial (\ref{PDEM}) becomes
\begin{equation}
\begin{split}
\label{div-pde}
& \left(\frac{v'_R(r)}{(\varrho^{-2}(r)+v'^2_R(r))^{1/2}}\right)'+\frac{v'_R(r)}{(\varrho^{-2}(r)+v'^2_R(r))^{1/2}}\left(\frac{\varrho'(r)}{\varrho(r)}+(n-1)\frac{\xi'(r)}{\xi(r)}\right) \\
& = nH(R).
\end{split}
\end{equation}
It follows from (\ref{div-int})  that  $v_R$ is the solution of the first order equation
\begin{equation} 
\label{ode-v}
\frac{v'_R(r)}{(\varrho^{-2}(r)+v'^2_R(r))^{1/2}} \,\varrho(r) \xi^{n-1}(r) =
\int_0^r nH(R) \varrho(\varsigma) \xi^{n-1}(\varsigma)\, {\rm d}\varsigma.
\end{equation}
with initial condition $v_R|_{r=R}=0$. 
Resolving this expression for $v'_R$, one obtains
\begin{equation}
\label{u'}
v'_R(r)=\frac{nH(R)V(r)}{\varrho(r)(A^2(r)-n^2 H^2(R)V^2(r))^{1/2}}\cdot
\end{equation}
The graph $\Sigma_R$ of $v_R$ is a rotationally invariant hypersurface which can be parametrized in terms of coordinates $(s,r,\vartheta)$
as $\varsigma\mapsto  (s(\varsigma), \vartheta, r(\varsigma))$, where $\varsigma$ can be taken as the arc lenght parameter. Denoting by $\phi$ the angle between the coordinate vector field $\partial_r$ and a given profile curve $\vartheta={\rm constant}$ in $\Sigma_R$ one has
\[
\dot r = \cos\phi, \quad \varrho \dot s = \sin\phi.
\]
Thus,  (\ref{div-pde}) becomes
\[
-\frac{{\rm d}}{{\rm d}\varsigma}(\varrho\dot s) \frac{{\rm d}\varsigma}{{\rm d}r}-\varrho\dot s\left(\frac{\varrho'(r)}{\varrho(r)}+(n-1)\frac{\xi'(r)}{\xi(r)}\right) = nH(R),
\]
that is,
\[
\frac{{\rm d}\phi}{{\rm d}\varsigma} +\sin\phi\left(\frac{\varrho'(r)}{\varrho(r)}+(n-1)\frac{\xi'(r)}{\xi(r)}\right) = -nH(R).
\]
Hence, a profile curve of $\Sigma_R$ is given by the solution of the first order system
\begin{align*}
\begin{cases}
& \dot r = \cos\phi,\\
& \varrho \dot s = \sin\phi,\\
& \dot \phi =- nH(R) - nH_{{\rm cyl}}(r)\sin\phi,
\end{cases}
\end{align*}
with initial conditions $r(0) = R, s(0) = 0, \phi(0) = \frac{\pi}{2}$. In this case (\ref{ode-v}) can be rewritten as
\[
\varrho(r) A(r)\dot s =
-nH(R) V(r)
\]
where $\cdot$ indicates derivatives with respect to the parameter $\varsigma$. Hence, it is immediate that when the coordinate $r$ attains its maximum, that is, when $r=R$, we have $\dot r=0$ and $\varrho \dot s=1$. This is consistent with the choice of $H(R)$ in 
(\ref{meanHR}). We also observe that $\dot s \to 0$ and $\dot r \to 1$ as $r \to 0^+$.  \hfill $\square$

\vspace{3mm}

For $R\ge r_0$, note that  the variable $\mu = R-r_0$ can be considered as the geodesic distance between the geodesic spheres $\partial B_{r_0}(o) = \partial\Sigma_{r_0}$ and $\partial B_R(o) =\partial\Sigma_R$. Hence, $\nabla \mu|_{\partial B_R(o)} =\partial_r|_{r=R}$.  Thus, we set a time parameter $t \in [0,\infty)$ given by
\begin{align} 
\label{RT}
& \frac{{\rm d}\mu}{{\rm d}t} = -nH(R) = -nH(\mu+r_0),\\
& \mu(0) = 0.
\end{align}
Hence, $\mu=\mu(t)$ is implicitly defined by

\begin{equation}\label{def-mu}
\int_{r_0}^{\mu(t)+r_0} \frac{V(\varsigma)}{A(\varsigma)}\,{\rm d}\varsigma = t
\end{equation}

Denote $R(t) = \mu(t) + r_0$. We claim that the one-parameter family of constant mean curvature graphs 
$\{\Sigma_{R(t)}\}_{t\ge 0}$ evolves by the  (negative) mean curvature flow 
\begin{equation}
\label{MCFtau}
\partial_t \Psi^+ = -nH(R(t))N_t,
\end{equation}
where
\[
N(t)= \frac{1}{W} (\varrho^{-2}(r) X - v'_R(r)\partial_r) = -\frac{\dot r}{\varrho} X +\varrho \dot s\, \partial_r.
\]
This means that $\Sigma_{R(t)} = \Psi_t^+ (\Sigma_{r_0})$. 
In particular, we must have
\[
\partial B_{R(t)} =\partial \Sigma_{R(t)} = \Psi^+_t (\partial\Sigma_{r_0}) = \Psi^+_t (\partial B_{r_0}).
\]
In other terms, the time parameter $t$ must be chosen in such a way that the geodesic spheres evolve as $\partial B_{R(t)}=\Psi^+_t (\partial B_{r_0})$. Since $\dot r =0$ and $\varrho \dot s=1$ at $r=R(t)$ it follows from (\ref{MCFtau})  that
\begin{equation}
	\begin{split}
		\frac{{\rm d}\mu}{{\rm d}t} &= \langle \partial_t\Psi^+, \nabla^+\mu\rangle = \langle \partial_t\Psi^+, \partial_r|_{r=R(t)}\rangle
		= -nH(R(t))\langle N_t, \partial_r\rangle|_{r=R(t)}\\
		&=  - nH(R(t)) = -nH(r_0+\mu(t))
	\end{split}
\end{equation}
what means that $t$ coincides with the parameter defined in (\ref{RT}) and then satisfying the condition that $\partial B_{R(t)}=\Psi_t^+ (\partial B_{r_0})$. Note that $R(t)\ge r_0$ for $t\ge 0$. We conclude that the one-parameter family of functions $u_+(x, t) = v_{R(t)}(r(x))$ defined on the common domain $B_{r_0}(o)$ defines a solution of the geometric flow (\ref{MCFtau}). Hence, we set
\begin{equation}
\Psi^+(x, t) = (x, u_+(x, t)), \quad x\in B_{r_0}(o).
\end{equation} 
It follows that $u_+$ satisfies the parabolic equation
\begin{equation}
\begin{split}
 \partial_t u_+ & =  -(\varrho^{-2}(r)+|\partial_r u_+|^2)^{1/2}  \bigg(\partial_r\left(\frac{\partial_r u_+}{(\varrho^{-2}(r)+|\partial_r u_+|^2)^{1/2}}\right)\\
& \quad  + \frac{\partial_r u_+}{(\varrho^{-2}(r)+|\partial_r u_+|^2)^{1/2}}\left(\frac{\varrho'(r)}{\varrho(r)}+(n-1)\frac{\xi'(r)}{\xi(r)}\right)\bigg).
\end{split}
\end{equation} \\
Now we will use this flow as a supersolution to the mean curvature flow in $M$.
\begin{proposition} 
	\label{supersolution}
		The one-parameter family of functions
	\begin{equation}
	\label{u+}
	u_+(x, t) = v_{R(t)} (x) = v_{R(t)}(r(x)), \quad x\in B_{r_0}(o), \quad t\in [0,\infty)
	\end{equation} 
	satisfies $\partial_t u_+ + \mathcal{Q}[u_+] \ge 0$ in $B_{r_0}(o)\times [0, T]$ for all $T>0$, hence is a supersolution of the mean curvature flow in $M=P\times_{\varrho} \mathbb{R}$.
\end{proposition}

\begin{proof}  Denoting $ W = (\varrho^{-2}+|\nabla^P u_+|^2)^{1/2} $ we have
	\begin{equation*}
	\begin{split}
	&  \mathcal{Q}[u_+] +\partial_t u_+= W\left({\rm div}_P\left(\frac{\nabla^P u_+}{W}\right) + \left\langle\nabla^P\log\varrho, \frac{\nabla^P u_+}{W}\right\rangle \right) +\partial_t u_+\\
	& \,\, =(\varrho^{-2}+ u'^2_+(r))^{1/2}\bigg(\partial_r\left(\frac{u'_+(r)}{(\varrho^{-2} + u'^2_+(r))^{1/2}}\right) \\
	& \,\,\,\,+ \frac{u'_+(r)}{(\varrho^{-2} + u'^2_+(r))^{1/2}} \big(\Delta_P r + \langle \nabla^P \log \varrho, \nabla^P r\rangle\big) \bigg)+ \partial_t u_+,
	\end{split}
	\end{equation*}
	where $\Delta_P$ is the Laplace-Beltrami operator in $(P, g)$. However
	\[
	\langle \nabla^P \log \varrho, \nabla^P r\rangle = \frac{\partial_r \varrho}{\varrho} = \frac{\varrho'(r)}{\varrho(r)}\cdot
	\]
	Moreover (\ref{hess-comp}) implies that 
	\[
	\Delta_P r \le (n-1)\frac{\xi'(r)}{\xi(r)}\cdot
	\]
	Since $u'_+= v'_R\le 0$ we conclude that
	\begin{equation*}
	\begin{split}
	&  \mathcal{Q}[u_+] +\partial_t u_+ \ge (\varrho^{-2}+u'^2_+(r))^{1/2}\bigg(\partial_r\left(\frac{u'_+(r)}{(\varrho^{-2} + u'^2_+(r))^{1/2}}\right)\\
	& \,\,\,\,+\frac{u'_+(r)}{(\varrho^{-2} + u'^2_+(r))^{1/2}}\left(\frac{\varrho'(r)}{\varrho(r)}+(n-1)\frac{\xi'(r)}{\xi(r)}\right) \bigg)   + \partial_t u_+=0.
	\end{split}
	\end{equation*}
	This finishes the proof. 
	\end{proof}

\begin{proposition}
	\label{height-mcf}
	 If $u$ be a solution of  {\rm(\ref{R-approximate-problem})}, then we have the following height estimate
	\begin{equation}
	|u(x,t)| \le {\sup}_{B_{r_0}(o)} |u| 
	+v_{R(T)} (o) - v_{r_0} (r(x)).
	\end{equation}
	More precisely,
	\begin{equation}
	\begin{split}
	& |u(x,t) |\le {\sup}_{B_{r_0}(o)} |u(\cdot, 0)| + \int^{0}_{R(T)} \frac{nH(R(T)) V(\varsigma)}{\varrho(\varsigma) (A^2(\varsigma)-n^2 H^2(R(T))V^2(\varsigma))^{\frac{1}{2}}}\,{\rm d}\varsigma\\
	& \,\, - \int^{r(x)}_{r_0} \frac{nH(r_0) V(\varsigma)}{\varrho(\varsigma) (A^2(\varsigma)-n^2 H^2(r_0)V^2(\varsigma))^{\frac{1}{2}}}\,{\rm d}\varsigma.
	\end{split}
	\end{equation}
	for $(x,t) \in B_{r_0}(o) \times [0, T]$.
\end{proposition}

\begin{proof} By construction, the graph $\Sigma_{r_0}$ of $u_+(\cdot, 0) =v_{r_0}$ is defined in the geodesic ball $B_{r_0}(o)$. Given $T>0$ we have that
	$\Psi^+_T(\Sigma_{r_0})$ is the graph $\Sigma_{R(T)}$ of $u_+ (\cdot, T) = v_{R_T}|_{B_{r_0}(o)}$ with
	\[
	\int_{r_0}^{R(T)} \frac{V(\varsigma)}{A(\varsigma)}\, {\rm d}\varsigma = T.
	\]
	Given $\varepsilon>0$ we have
	\[
	-u_+(x,T) + u_+(o,T)+ {\sup}_{B_{r_0}(o)} u + \varepsilon > u(x, 0)
	\] 
	for all $x\in B_{r_0}(o)$.  We also have
	\[
	v_\varepsilon(x, t) := -u_+(x, T-t) + u_+(o,T)+ {\sup}_{B_{r_0}(o)} u + \varepsilon > u(x, t)
	\]
	for all $(x, t) \in \partial B_{r_0} (o) \times [0, T]$. Proposition \ref{supersolution} implies that
	\begin{equation}
	\partial_t v_\varepsilon - \mathcal{Q}[v_\varepsilon] =\partial_t u_+ + \mathcal{Q}[u_+] \ge 0
	\end{equation}
	in the parabolic cylinder $B_{r_0}(o)\times (0, T)$. Then the parabolic maximum principle implies that
	\[
	u(x,t) \le v (x,t) \le v(x,T)
	\]
	in $B_{r_0} (o) \times [0, T]$ where
	\begin{equation}
	v(x, t) = -u_+(x, T-t) + u_+(o,T)+ {\sup}_{B_{r_0}(o)} u.
	\end{equation}
	Hence
	\begin{equation*}
	u(x,t) \le v(x,T) = u_+(o,T)-u_+(x, 0) + {\sup}_{B_{r_0}(o)} u.
	\end{equation*}
	Therefore
	\begin{equation}
	\label{supersol}
	u(x,t) \le {\sup}_{B_{r_0}(o)} u 
	+v_{R(T)} (o) - v_{r_0} (r(x))
	\end{equation}
	for $(x,t) \in B_{r_0} (o) \times [0, T]$.  We prove in a similar way that 
	\[
	u(x,t) \ge  w (x,t) \ge w(x,T)
	\]
	in $B_{r_0} (o) \times [0, T]$ where
	\begin{equation}
	w (x, t) = u_+(x, T-t) - u_+(o,T)+ {\inf}_{B_{r_0}(o)} u.
	\end{equation}
	Hence
	\begin{equation}
	\label{subsol}
	u(x,t) \ge {\inf}_{B_{r_0}(o)} u -
	v_{R(T)} (o) + v_{r_0} (r(x))
	\end{equation}
	in $B_{r_0} (o) \times [0, T]$. This finishes the proof.
\end{proof}

\vspace{3mm}

Now we obtain an uniform estimate for $C^0$ bounds of the   functions $u_+( \cdot, t).$
 
 \begin{proposition}\label{estimate-u_+}
 	Let $r_0 >0$ be a fixed constant and $R(t)\ge  r_0$, $t\in[0, \infty)$, 
	be the function implicitly defined in $(\ref{def-mu})$.
	If $\ell_0 > 0$ satisfies 
	\[
	R(t) \leq \ell_0 r_0 \quad \mbox {for  all}  \quad  t \in [0,T]
	\]
	 then 
 	\[
 	{\sup}_{B_{r_0}(o) \times [0, T]} u_+(x,t) = {\sup}_{[0,T]}u_+(o, t) \leq c(r_0, \, \ell_0).
 	\] 
 \end{proposition}

\begin{proof} It follows directly from \eqref{vR} and \eqref{u+} that
	\begin{align*}
	  & u_+(x,t) = v_{R(t)}(r(x)) = \int_{r(x)}^{R(t)}\frac{- n H(R(t))V(\varsigma)}{\varrho(\varsigma) \big(A^2(\varsigma) - n^2H^2(R(t))V^2(\varsigma)\big)^{\frac{1}{2}}}{\rm d}\varsigma \\
	 & \leq \int_{0}^{R(t)}\frac{- n H(R(t))V(\varsigma)}{\varrho(\varsigma) \big(A^2(\varsigma) - n^2H^2(R(t))V^2(\varsigma)\big)^{\frac{1}{2}}} {\rm d}\varsigma = v_{R(t)}(o).
	\end{align*}
	Moreover	
	\begin{align*}
	 & \frac{- n H(R(t))V(\varsigma)}{\varrho(\varsigma) \big(A^2(\varsigma) - n^2H^2(R(t))V^2(\varsigma)\big)^{\frac{1}{2}}} = \frac{- n H(R(t))V(\varsigma)}{- H(R(t))A(\varsigma)\varrho(\varsigma) \Big(\frac{1}{H^2(R(t))} - \frac{1}{H^2(\varsigma)} \Big)^{\frac{1}{2}}} \\
	&= - \frac{1}{H(\varsigma) \varrho(\varsigma)} \bigg(\frac{1}{H^2(R(t))} - \frac{1}{H^2(\varsigma)} \bigg)^{-\frac{1}{2}} 
    \\
    &= - \frac{H^2}{\varrho H'} \frac{H'}{H^3} \bigg(\frac{1}{H^2(R(t))} - \frac{1}{H^2}(\varsigma) \bigg)^{-\frac{1}{2}} \\
	& \leq -{\sup}_{[0,\ell_0 r_0]} \bigg(  \frac{H^2}{\varrho H'}\bigg) \frac{H'}{H^3} \bigg(\frac{1}{H^2(R(t))} - \frac{1}{H^2(\varsigma)} \bigg)^{-\frac{1}{2}},
	\end{align*}
	where in the right hand side of the inequality above we use that $nH(r) = - \frac{A(r)}{V(r)}$ is an increasing function. Note that 
	\[
	nH'
	= -\frac{A'}{A} \frac{A}{V} + \frac{A}{V} \frac{V'}{V}= -nH\bigg(\frac{V'}{V}- \frac{A'}{A}\bigg).
	\]
	Therefore
	\[
	\frac{\varrho H'}{H^2} = -\frac{\varrho}{H}\bigg(\frac{V'}{V}- \frac{A'}{A}\bigg)
	\]
	Using the change of variables 
	\[
	\eta = \bigg(\frac{1}{H^2(R(t))} - \frac{1}{H^2(\varsigma)}\bigg)^{\frac{1}{2}}
	\] 
	one gets
	\begin{align*}
	& v_{R(t)}(o) 
	\leq - \bigg( \sup_{[0,\ell_0 r_0]} \frac{H^2}{\varrho H'}\bigg) \int_{0}^{R(t)} \frac{H'}{H^3} \bigg(\frac{1}{H^2(R(t))} - \frac{1}{H^2(\varsigma)} \bigg)^{-\frac{1}{2}} {\rm d} \varsigma \\
	& = - \bigg( \sup_{[0,\ell_0 r_0]} \frac{H^2}{\varrho H'}\bigg) \int_{- \frac{1}{H(R(t))}}^{0} {\rm d} \eta \\
	&= -  \sup_{[0,\ell_0 r_0]} \bigg(\frac{H^2}{\varrho H'}\bigg) \frac{1}{H(R(t))} 
    \\
    &\leq  - \bigg( \sup_{[0,\ell_0 r_0]} \frac{H^2}{\varrho H'}\bigg) \frac{1}{H(\ell_0 r_0)}.
	\end{align*}
Thus we have
	\[
	\sup_{B_{r_0}(o) \times [0, T]}|u_+(x,t)| = \sup_{[0,T]}u_+(o, t) \leq c(r_0, \, \ell_0, \, \varrho, \, \xi).
	\] 

\end{proof}

A consequence of this proposition is an a priori height estimate that does not depend on the maximum time of the solution.

\begin{corollary}\label{unif-height-est}
	Let $u$ be a solution of \eqref{R-approximate-problem} in $B_{r_0}\times [0, \epsilon]$ and $R_0 : [0, \infty) \rightarrow [r_0, \infty)$ be the function implicitly defined in $(\ref{def-mu}).$ For $\tau > \epsilon$, if $\ell_0 > 0$ satisfies $$R_0(t) \leq \ell_0 r_0 \qquad \forall \quad t \in [0,\tau]$$ then 
	$$ |u(x,t)| \le {\sup}_{B_{r_0}(o)} |u| + c(r_0, \tau, \ell_0, \varrho, \xi) - v_{r_0} (r(x)).$$
\end{corollary}

\begin{proof}
	In fact, for $(x,t) \in B_{r_0}\times[0,\epsilon],$ we have
	\begin{align*}
	|u(x,t)| &\le {\sup}_{B_{r_0}(o)} |u| + v_{R(\epsilon)} (o) - v_{r_0} (r(x)) \\
	&\le {\sup}_{B_{r_0}(o)} |u| 	+ |u_+(o,\epsilon)|- v_{r_0} (r(x))\\
	&\le {\sup}_{B_{r_0}(o)} |u| 	+ c(r_0, \tau, \ell_0, \varrho, \xi)  - v_{r_0} (r(x))
    \end{align*}
\end{proof}

\subsection{Boundary gradient estimate}
In this part we obtain a boundary gradient estimate and a quantitative interior gradient estimate for a solution of $\rm{(\ref{R-approximate-problem})}.$ More precisely, we obtain a quantitative interior gradient estimate in a cylinder contained in $B_R(o)\times [0,T]$ which will be very important in the proof of the Theorem $\ref{main-2}$.

\begin{proposition}\label{boundary-grad-est}
		Let $u$ be a  solution of  \eqref{R-approximate-problem} defined in  $B_R(o)\times [0, T]$ for $R>0$ and $T>0$. Then there exists a constant $C >0$ such that
		\[
		\sup_{\partial B_R(o) \times [0, T]} |\nabla u| \leq C.
		\]
\end{proposition}

\begin{proof}
    For each $x\in B_R(o)$, consider the function $d(x)={\rm dist}(x,\partial B_R(o)).$ Observe that, we can rewrite the function $d$ as $d(x)=R-r(x)$. In order to guarantee the differentiability of $r$, consider $\varepsilon>0$ sufficient small, such that there is no focal points of $\partial B_R(0)$ for $r\in (R-\varepsilon,R+\varepsilon)$.   
    
    Let $\widetilde{u}_0$ be a local extension of the function $u_0$ defined in problem $\ref{main-problema}$. We then define a funtion $v$ as 
    \begin{equation}
        v(x)=\widetilde{u}_0(x)+h(d(x)),
    \end{equation}
    where $h$ is an arbitrary, \textit{a priori}, twice differentiable function. Hence, by direct computation,
    \begin{equation}
        W = \sqrt{\varrho^{-2} + |\nabla^P v|^2} = \sqrt{\varrho^{-2}+ h'^2(d) + 2h'(d)\langle\nabla^P d, \nabla^P \widetilde u_0\rangle + |\nabla^P \widetilde u_0|^2},
    \end{equation}
    and then after some reorder in the terms we get that
    \begin{equation*}
        \begin{split}
            &  W^2 (\partial_t v - \mathcal{Q}[v])= - h''(d)  W^2 + h '' (d) (\widetilde u_0^i + h'(d) d^i ) (\widetilde u_0^j + h'(d) d^j)  d_i d_j \\
            & \,\, -W^2  (\Delta_P \widetilde u_0 + h' (d) \Delta_P d)  + (\widetilde u_0^i + h'(d) d^i ) (\widetilde u_0^j + h'(d) d^j) (\langle\nabla^P_{\partial_i}\nabla^P \widetilde u_0, \partial_j\rangle \\
            & \,\, + h' (d)\langle \nabla^P_{\partial_i} \nabla^P d, \partial_j\rangle) - \langle \nabla^P \log\varrho, \nabla^P  \widetilde u_0 + h' (d) \nabla^P d\rangle (\varrho^{-2}+ W^2).
        \end{split}
    \end{equation*}
    Since $|\nabla^P d|=1$ and $d^i d^j  \langle \nabla^P_{\partial_i} \nabla^P d, \partial_j\rangle = 0$ we also have
    \begin{equation*}
        \begin{split}
            & -W^2  (\Delta_P \widetilde u_0 + h' (d) \Delta_P d)  + (\widetilde u_0^i + h'(d) d^i ) (\widetilde u_0^j + h'(d) d^j) (\langle\nabla^P_{\partial_i}\nabla^P \widetilde u_0, \partial_j\rangle \\
            &+ h' (d)\langle \nabla^P_{\partial_i} \nabla^P d, \partial_j\rangle)\\
            &= -(\varrho^{-2} + h'^2 (d) ) h'(d) \Delta_P d - h'^2(d) (\Delta_P \widetilde u_0 
            \\
            &+ 2\langle \nabla^P d, \nabla^P \widetilde u_0\rangle \Delta_P d - d^i d^j \langle\nabla^P_{\partial_i}\nabla^P \widetilde u_0, \partial_j\rangle )\\
            & \,\,\,\, -h' (d) ( |\nabla^P \widetilde u_0|^2 \Delta_P d - \widetilde u_0^i \widetilde u^j_0 \langle \nabla^P_{\partial_i} \nabla^P d, \partial_j\rangle
            \\
            &+2\langle \nabla^P d, \nabla^P \widetilde u_0\rangle \Delta_P \widetilde u_0   - 2 d^i \widetilde u^j_0 \langle \nabla^P_{\partial_i} \nabla^P \widetilde u_0, \partial_j\rangle)\\
            & \,\,\,\, - \varrho^{-2} \Delta_P \widetilde u_0 - |\nabla^P \widetilde u_0^2| \Delta_P \widetilde u_0 + \widetilde u_0^i \widetilde u^j_0 \langle \nabla^P_{\partial_i} \nabla^P \widetilde u_0, \partial_j\rangle.
        \end{split}
    \end{equation*}
Gathering these expressions and requiring that  $h'>0$ and $h ''< 0$, one gets
\begin{equation*}
\begin{split}
  &W^2 (\partial_t v - \mathcal{Q}[v]) \ge - h''(d)  \varrho^{-2} -(\varrho^{-2} + h'^2 (d) ) h'(d) \Delta_P d\\
& \,\,- \langle \nabla^P \log\varrho, \nabla^P  \widetilde u_0 + h' (d) \nabla^P d\rangle (\varrho^{-2}+ W^2)  \\
& \,\,- h'^2(d) (\Delta_P \widetilde u_0 + 2|\nabla^P \widetilde u_0| |\Delta_P d| + | \nabla^P \nabla^P\widetilde u_0|)\\
& \,\, -h' (d) ( |\nabla^P \widetilde u_0|^2 |\Delta_P d| +  |\nabla^P \widetilde u_0|^2 |\nabla^P \nabla^P d|+2|\nabla^P \widetilde u_0| |\Delta_P \widetilde u_0| \\ 
& \,\,  + 2 |\nabla^P\widetilde u_0| | \nabla^P \nabla^P \widetilde u_0|)
- C(\varrho^{-2} + |\nabla^P \widetilde u_0|^2)  |\nabla^P  \nabla^P \widetilde u_0|,
\end{split}
\end{equation*}
where $C$ here and in what follows stands for a positive constant that depends on $n$ and on the first and second derivatives of $d$.  Hence, 
\begin{equation*}
\begin{split}
&  W^2 (\partial_t v - \mathcal{Q}[v])\ge - h''(d)  \varrho^{-2} -(\varrho^{-2} + h'^2 (d) ) h'(d) \Delta_P d\\
& \,\,- \langle \nabla^P \log\varrho, \nabla^P  \widetilde u_0 + h' (d) \nabla^P d\rangle (\varrho^{-2}+ W^2)\\
& \,\,- Ch'^2(d) ( |\nabla^P \widetilde u_0| + |\nabla^P \nabla^P\widetilde u_0| )
\\
& \,\, -Ch' (d) \big(   |\nabla^P \widetilde u_0|^2  +  |\nabla^P\widetilde u_0| | \nabla^P \nabla^P \widetilde u_0| \big) - C(\varrho^{-2} + |\nabla^P \widetilde u_0|^2)  |\nabla^P  \nabla^P \widetilde u_0| \\
& \,\, =  - h''(d)  \varrho^{-2} - h'(d) \langle \nabla^P \log \varrho, \nabla^P d \rangle \varrho^{-2}    \\
& \,\,-(\varrho^{-2}+ h'^2 (d) ) h'(d) \big(\Delta_P d+ \langle \nabla^P \log \varrho, \nabla^P d \rangle \big)
\\
&\,\,-h'(d) \langle \nabla^P \log \varrho, \nabla^P d \rangle \big(|\nabla^P\widetilde u_0|^2 + 2h'(d) \langle \nabla^P d, \nabla^P\widetilde u_0 \rangle \big)\\
& \,\,  - (\varrho^{-2} + W^2)\langle \nabla^P \log \varrho, \nabla^P \widetilde u_0 \rangle - Ch'^2(d) ( |\nabla^P \widetilde u_0| + |\nabla^P \nabla^P\widetilde u_0| )
\\
&\,\, -Ch' (d) \big(   |\nabla^P \widetilde u_0|^2  +  |\nabla^P\widetilde u_0| | \nabla^P \nabla^P \widetilde u_0| \big) \,\,  - C(\varrho^{-2} + |\nabla^P \widetilde u_0|^2)  |\nabla^P  \nabla^P \widetilde u_0|. 
\end{split}
\end{equation*}
It follows from (\ref{hess-comp}) and (\ref{cylinder-0}) that
\begin{equation*}
\begin{split}
- \big( \Delta_P d + \langle \nabla^P \log \varrho, \nabla^P d \rangle \big) = \Delta_Pr + \langle \nabla^P \log \varrho, \nabla^P r \rangle \geq -n \frac{\xi'(r)}{\xi(r)} \geq -nB,
\end{split}
\end{equation*}
where $B:= \sup_{B_{R}(o)}\frac{\xi'(r(x))}{\xi(r(x))}.$

Therefore 
\begin{equation*}
\begin{split}
  &W^2 (\partial_t v - \mathcal{Q}[v]) \ge  - h''(d)  \varrho^{-2}- h'(d)| \nabla^P \log\varrho|  \varrho^{-2} - nBh'^3(d) \\
&-Ch'^2(d)\bigg(|\nabla^P\log \varrho||\nabla^P\widetilde u_0| + |\nabla^P\widetilde u_0| + |\nabla^P \nabla^P \widetilde u_0|\bigg)\\
& - C h'(d)\bigg( nB \varrho^{-2} + |\nabla^P \log\varrho||\nabla^P\widetilde u_0|^2 + |\nabla^P\widetilde u_0|^2 + |\nabla^P \widetilde u_0||\nabla^P \nabla^P \widetilde u_0| \bigg)\\
& - C\big(\varrho^{-2} + |\nabla^M \widetilde u_0|^2 \big)  \bigg( |\nabla^P \log \varrho||\nabla^P \widetilde u_0| + |\nabla^M  \nabla^M \widetilde u_0| \bigg).
\end{split}
\end{equation*}

For $L>0$, we fix $d_0 < \frac{1}{L}$  and $d_0< d^*$ and take
\[
A = \frac{L}{1-Ld_0}\cdot
\]
We consider
\[
h(d) = \frac{1}{L} \log (1+Ad) 
\]
for $d\in [0,d_0]$ and we note that
\[
h'(d) = \frac{1}{L}\frac{A}{1+Ad}, \quad \mbox{and} \quad h''(d)= -Lh'^2 (d).
\]

Thus 
\begin{equation*}
\begin{split}
W^2 (\partial_t v - \mathcal{Q}[v])& \ge   \bigg(\frac{1}{L} \frac{A^2}{(1+Ad)^2}  - \frac{1}{L}\frac{A}{1+Ad} |\nabla^M \log\varrho| \bigg)\varrho^{-2} \\
&-nB\frac{1}{L^3}\frac{A^3}{(1 + Ad)^3} - \widetilde L \bigg(\frac{1}{L^2} \frac{A^2}{(1+Ad)^2} + \frac{1}{L}\frac{A}{1+Ad}+1 \bigg),
\end{split}
\end{equation*}
where
\[
\widetilde L  = C\bigg(n, B, \varrho, |\nabla^P\log \varrho|, |\nabla^P  \widetilde u_0|, |\nabla^P\nabla^P  \widetilde u_0| \bigg).
\]
More precisely,
\begin{equation*}
\begin{split}
  W^2 (\partial_t v - \mathcal{Q}[v])&\ge  \frac{1}{L} \bigg\{ -nB\frac{1}{L^2}\frac{A^3}{(1+ Ad)^3} +  \bigg(\varrho^{-2} - \frac{\widetilde L}{L} \bigg) \frac{A^2}{(1 + Ad)^2} \\
& \,\,\qquad - \bigg(|\nabla^P \log \varrho|\varrho^{-2} + \widetilde L\bigg) \frac{A}{1 + Ad}  - L \widetilde L  \bigg\},
\end{split}
\end{equation*}
For $d \in [0, d_0]$ we have
\[
L = \frac{A}{1 + Ad_0} \leq \frac{A}{1 + Ad} \quad \mbox{and} \quad - L \widetilde L \ge - \widetilde L \frac{A}{1 + Ad}.
\]
So 
\begin{equation*}
\begin{split}
W^2 (\partial_t v - \mathcal{Q}[v])&\ge  \frac{1}{L} \frac{A}{1 + Ad}\bigg\{ -nB\frac{1}{L^2}\frac{A^2}{(1+ Ad)^2} +  \bigg(\varrho^{-2} - \frac{\widetilde L}{L} \bigg) \frac{A}{1 + Ad} \\
& \qquad - \bigg(|\nabla^P \log \varrho|\varrho^{-2} + 2 \widetilde L \bigg)  \bigg\}\\
& \ge \frac{\widetilde L}{L} \frac{A}{1 + Ad} \bigg \{ - \frac{1}{L^2}\frac{A^2}{(1 + Ad)^2} + \bigg( \frac{L}{\varrho^2 \widetilde L} - 1\bigg)\frac{1}{L} \frac{A}{1 + Ad} - 3\bigg \}
\end{split}
\end{equation*}

If we choose $L > 10 \widetilde L^2$ we have 
\[
\Delta = L^2 \bigg( \frac{L}{\varrho^2 \widetilde L} - 1\bigg)^2 - 12L^2 > 0.
\]
As $\Delta$ is the discriminant of the inequality
$$ - z^2 + \bigg(\frac{L}{\varrho^2 \widetilde L} - 1\bigg) Lz -3 L^2 \ge 0$$
we can choose $d_0 < \frac{1}{L}$ such that 
\[
 - A^2 + \bigg(\frac{L}{\varrho^2 \widetilde L} - 1\bigg) LA -3 L^2 \ge 0.
\]
Then we get

\begin{equation*}
\begin{split}
W^2 (\partial_t v - \mathcal{Q}[v])& \ge \frac{\widetilde L}{L} \frac{A}{1 + Ad} \bigg \{ - \frac{1}{L^2}\frac{A^2}{(1 + Ad)^2} + \bigg( \frac{L}{\varrho^2 \widetilde L} - 1\bigg)\frac{1}{L} \frac{A}{1 + Ad} - 3\bigg \}
\end{split}
\end{equation*}
is nonnegative for all $(x, t) $ with $d(x) \in [0, d_0].$ \\
Hence $ v = \widetilde{u_0} + h(d)$ is an upper barrier. If we take $ \omega =- \widetilde{u_0} -  h(d)$ we get a lower barrier. Therefore  there exists a constant $C >0 $ such that
\[
\sup_{\partial B_R(o) \times [0, T]} |\nabla u| \leq C.
\]
\end{proof}

\subsection{Interior gradient estimate}

In this section we will use a technique due to Korevaar and Simon \cite{korevaar},	and further developed by Wang \cite{Wang:98}  to prove an interior gradient estimates.

Given $R > 0, \,\mbox{and} \, T>0,$ let 
$$\mathcal{C}_{R,T} = \big\{ \Psi(x,t) ; \, \, \zeta(r(\Psi(x,t))) + t < \zeta(R), \, \, x \in B_R(o), \, \, t \in [0,T]\big\}.$$ 
If $R' \in (0,R)$ is such that $\zeta(r) < \frac{1}{4}\zeta(R)$ for all $r\leq R'$ we get 
$B_{R'}(o)\times[0,T_{R'}] \subset \mathcal{C}_{R,T}$ with $ T_{R'} = \min \left\{\frac{1}{2}\zeta(R), T \right\}.$ 

\begin{proposition}\label{int-grad-est}
	Let $u$ be a positive solution of \eqref{R-approximate-problem} defined in  $B_R(o)\times [0, T]$ for $R>0$ and $T>0$ . Let  $L \geq 0$ be a constant such that ${\rm Ric}_g - \nabla^2 \log \varrho \ge - Lg$ in $B_{R}(o)$ and suppose that \eqref{K-comp} holds. Then for $(x,t)$ in $B_{R'}(o)\times[0,T_{R'}]$
	either
	\begin{equation}
	\begin{split}
	|\nabla^P u(x,t)| \le \exp \left(128 \frac{ \left(1 + \min_{\overline{B}_R(o)}\varrho\right)^2 }{\min_{\overline{B}_R(o)} \varrho}M \sup_{B_{R}(o)}\frac{\xi(r)}{\zeta(R)} \right)
	\end{split}
	\end{equation}
    or
		\begin{equation}
	\begin{split}
	\label{grad-est-1}
	|\nabla^P u(x,t)| \le \exp \left(64 \frac{ \left(1 + \min_{\overline{B}_R(o)}\varrho\right)^2 }{\min_{\overline{B}_R(o)} \varrho}M C_0 \right), \nonumber
	\end{split}
	\end{equation}
		where $M = \sup_{B_R(p)\times [0, T]} u, $ \,  and
		\begin{equation} 
            \text{\small{$C_0 = \sup_{B_R(o)} \frac{\varrho^2}{\mu} \biggl\{\frac{5}{4} + nM\frac{\xi'(r)}{\zeta(R)} + 2\sqrt{1-\beta}\frac{\xi(r)}{\zeta(R)} 
		      + \biggl(M(6 - 5\beta)\frac{\xi(r)}{\zeta(R)}+ 2\sqrt{1-\beta}\biggr) \frac{\varrho'}{\varrho}\biggr\}$.}}
        \end{equation}
	
\end{proposition}

\begin{proof} Suppose initially that  $u$ is  $C^3$ positive solution of the equation \eqref{R-approximate-problem} in  $B_R(o) \times (0, T) \subset P \times \mathbb{R}$. 
	We consider a nonnegative and smooth function $\eta$ with $\eta=0$ in $P\times\mathbb{R}\setminus B_R(o)\times \mathbb{R}$ and define a function $\chi$ in $\overline{B_R(o)} \times [0,T]$ of the form 
	\begin{equation}\label{chi-def}
	\chi = \eta \gamma(u) \psi (|\nabla u|^2),
	\end{equation}
	where the functions $\eta$, $\gamma$ and $\psi$ will be specified later.

	Suppose that $\chi$ attains its maximum in $\mathcal{C}_{R,T}$ at point  $(x_0, t_0)$, and without loss of generality, we suppose that $\eta(x_0, t_0)\neq 0$. 
	Then at $(x_0, t_0)$
	\begin{equation}\label{Foc-1}
	(\log \chi)_j = \frac{\eta_j}{\eta} + \frac{\gamma'}{\gamma} u_j + 2\frac{\psi'}{\psi} u^k u_{k;j} =0
	\end{equation}
	and therefore
	\begin{equation}\label{reduction}
	2\frac{\psi'}{\psi}u^k u_{k;j} = -\bigg(\frac{\eta_j}{\eta} + \frac{\gamma'}{\gamma} u_j\bigg).
	\end{equation}
	Moreover, the matrix
	\begin{equation}
    \begin{split}
	(\log\chi)_{i;j} &= (\log \eta)_{i;j} +\bigg(\frac{\gamma'}{\gamma} \bigg)' u_iu_j +\frac{\gamma'}{\gamma} u_{i;j}
    \\
    &+ 2\frac{\psi'}{\psi} (u^k u_{k;ij}+ u^k_{;i} u_{k;j}) 
	+ 4\bigg(\frac{\psi'}{\psi}\bigg)' u^k u_{k;i} u^\ell u_{\ell;j}
    \end{split}
	\end{equation}
	is non-positive at $(x_0, t_0)$. Applying the Ricci identities for the Hessian of $u$ we have
	\[
	u_{k;ij} = u_{i;kj} = u_{i;jk} + R^\ell_{kji}u_\ell,
	\]
	and this yields
	\begin{align*}\label{log-matr}
	(\log\chi)_{i;j} 
	&\,\,= \frac{\eta_{i;j}}{\eta} +\frac{\gamma''}{\gamma}  u_iu_j +\frac{\gamma'}{\gamma} u_{i;j} + \frac{\gamma'}{\gamma}\bigg(\frac{\eta_i}{\eta} u_j + \frac{\eta_j}{\eta}u_i\bigg) \\
	&\,\,\,\,+ 2\frac{\psi'}{\psi} (u^k u_{i;jk}+ u^k_{;i} u_{k;j}) -2\frac{\psi'}{\psi}R_{jki}^\ell u^k u_\ell 
    \\
    &+ 
	4\bigg(\bigg(\frac{\psi'}{\psi}\bigg)' - \frac{\psi'^2}{\psi^2}\bigg) u^k  u_{k;i}  u^\ell u_{\ell;j}.   \nonumber
	\end{align*}
	On the other hand, denoting
	\begin{equation}
	\label{f(x)}
	f(x) = \partial_tu - \big< \bar{\nabla}\log \varrho, \bar{\nabla}u \big > \bigg(\ 1 + \frac{1}{\varrho ^2 W^2}\bigg) 
	\end{equation}
	and differentiating both sides in \eqref{pde-q} we have
	\begin{equation}
	\label{sigma-u-ijk}
	\sigma^{ij} u_{i;jk} = f_k - \sigma^{ij}_{;k} u_{i;j}.
	\end{equation}
	Contracting \eqref{sigma-u-ijk} with $u^k$, we get
	\begin{align*}
	\sigma^{ij} u^k u_{i;jk}  &= f_k u^k  +\frac{1}{W^2} u^k(u^i_{;k} u^j + u^i u^j_{;k})  u_{i;j} \\
	\quad& - \frac{2}{W^4}  u^i u^j u_{i;j} \big(- \varrho^{-2}(\log \varrho)_k u^k + u^k u^\ell u_{\ell;k} \big).
	\end{align*}
	Using the previous identity, \eqref{reduction} 
	and noticing that 
	\begin{align*}
	\sigma^{ij} R_{jki}^\ell u^k u_\ell 
	&= -{\rm Ric}_g (\nabla^P u, \nabla^P u),
	\end{align*}
	some simple computations give
	\begin{align*}
	0 &\ge \sigma^{ij}(\log \chi)_{i;j}=\sigma^{ij} \frac{\eta_{i;j}}{\eta}+\frac{\gamma''}{\gamma} \sigma^{ij}u_i u_j+ \frac{\gamma'}{\gamma}\partial_t u 
    \\
    &- \frac{\gamma'}{\gamma}\left( 1 + \frac{1}{\varrho^2 W^2}\right) \big \langle\nabla^P \log \varrho, \nabla^P u \big \rangle  \\ 
	&+2\frac{\gamma'}{\gamma}\frac{1}{\varrho^2W^2}
	\bigg\langle\frac{\nabla^P\eta}{\eta}, \nabla^P u\bigg\rangle 
	+ \quad 2\frac{\psi'}{\psi} u^k \partial_t u_k + \frac{4|\nabla^P u|^2}{\varrho^2 W^2} \frac{\psi'}{\psi}{\big \langle\nabla^P \log \varrho, \nabla^P u \big \rangle}^2 \\
	& - \frac{4}{\varrho^2 W^4} \bigg \langle \frac{\nabla^P \eta}{\eta} + \frac{\gamma'}{\gamma}\nabla^P u, \nabla^P u\bigg \rangle \big \langle\nabla^P \log \varrho, \nabla^P u \big \rangle 
    \\
    &- 2 \frac{\psi'}{\psi}\left( 1 + \frac{1}{\varrho^2 W^2}\right) \nabla^2 \log \varrho(\nabla^P u, \nabla^P u) \\
	& \bigg( 1 + \frac{1}{\varrho^2 W^2}\bigg) \bigg \langle \nabla^P \log \varrho, \frac{\nabla^P \eta}{\eta} + \frac{\gamma'}{\gamma} \nabla^P u\bigg \rangle  
    \\
    &+ 4 \bigg(\bigg(\frac{\psi'}{\psi}\bigg)' - \frac{\psi'^2}{\psi^2} + \frac{3}{2}\frac{1}{W^2}\frac{\psi'}{\psi}\bigg) \sigma^{i \ell}u^j u^k u_{k;i} u_{j; \ell}\\
	&+ 2\frac{\psi'}{\psi} {\rm Ric}_g (\nabla^P u, \nabla^P u) 
	+2\frac{\psi'}{\psi}\sigma^{i\ell} \sigma^{jk} u_{k;i} u_{j;\ell}.
	\end{align*}
	On the other  hand one has
	\[
	\partial_t \log\chi = \frac{\partial_t\eta}{\eta}+\frac{\gamma'}{\gamma}\partial_t u +2\frac{\psi'}{\psi} u^k \partial_t u_k.
	\]
	Hence
	\begin{align*}
	0 &\le \partial_t \log\chi- \sigma^{ij}(\log \chi)_{i;j} =\frac{\partial_t\eta}{\eta}-\sigma^{ij} \frac{\eta_{i;j}}{\eta}-\frac{\gamma''}{\gamma} \sigma^{ij}u_i u_j 
    \\
    &+ \frac{\gamma'}{\gamma}\bigg(1 + \frac{1}{\varrho^2 W^2}\bigg)\big \langle\nabla^P \log\varrho, \nabla^P u \big \rangle -2\frac{\gamma'}{\gamma}\frac{1}{\varrho ^2 W^2}\bigg\langle\frac{\nabla^P \eta}{\eta}, \nabla^P u\bigg\rangle 
	\\
    &- \frac{4|\nabla^P u|^2}{\varrho^2 W^4}\frac{\psi'}{\psi}\big \langle \nabla^P \log \varrho, \nabla^P u \big \rangle ^2 
    \\
    &+ \frac{4}{\varrho^2 W^4} \bigg \langle \frac{\nabla^P \eta}{\eta} + \frac{\gamma'}{\gamma}\nabla^P u, \nabla^P u \bigg \rangle \big \langle \nabla^P \log \varrho, \nabla^P u \big \rangle \\
	&+ 2 \frac{\psi'}{\psi} \bigg(1 + \frac{1}{\varrho^2 W^2}\bigg) \nabla^2 \log \varrho (\nabla^P u, \nabla^P u) 
    \\
    &- \bigg(1 + \frac{1}{\varrho^2 W^2}\bigg) \bigg \langle \nabla^P \log \varrho, \frac{\nabla^P \eta}{\eta} + \frac{\gamma'}{\gamma} \nabla^P u \bigg \rangle \\
	& - 4 \bigg[ \bigg(\frac{\psi'}{\psi}\bigg)' - \frac{{\psi'}^2}{\psi^2} + \frac{3}{2W^2}\frac{\psi'}{\psi}\bigg]\sigma^{i\ell}u^j u^k u_{k;i}u_{j;\ell} 
    \\
    &- 2\frac{\psi'}{\psi}\sigma^{i\ell} \sigma^{jk}u_{k;i}u_{j;l} -2\frac{\psi'}{\psi} {\rm Ric}_g (\nabla^P u, \nabla^P u)
	\end{align*}
	Note that 
	\begin{align*}
	4\frac{\psi'^2}{\psi^2}\sigma^{i\ell}  u^j u^k  u_{k;i}   u_{j;\ell} &= \bigg|\frac{\nabla^P \eta}{\eta} 
	+\frac{\gamma'}{\gamma}\nabla^P u\bigg|^2
	\ge \frac{1}{\varrho^2 W^2} \bigg|\frac{\nabla^P \eta}{\eta}
	+\frac{\gamma'}{\gamma}\nabla^P u\bigg|^2_g \\
	& =\frac{|\nabla^P u|^2}{\varrho^2 W^2} \bigg|\frac{\nabla^P \eta}{|\nabla^P u|\eta}
	+ \frac{\gamma'}{\gamma}\frac{\nabla^P u}{|\nabla^P u|}\bigg|^2_g.
	\end{align*}
	
	Plugging this into the previous estimate one obtains
	\begin{align*}
	& \frac{\psi^2}{\psi'^2} \bigg(\bigg(\frac{\psi'}{\psi}\bigg)' - \frac{\psi'^2}{\psi^2} + \frac{3}{2W^2}\frac{\psi'}{\psi}
	\bigg) \frac{|\nabla^P u|^2}{\varrho^2 W^2} 
	\bigg|\frac{\nabla^P \eta}{|\nabla^P u|\eta} 
    \\
    &+\frac{\gamma'}{\gamma}\frac{\nabla^P u}{|\nabla^P u|}\bigg|^2_g +\frac{\gamma''}{\gamma} \sigma^{ij}u_i u_j
	+2\frac{\psi'}{\psi}\sigma^{i\ell} \sigma^{jk} u_{k;i} u_{j;\ell}  \\
	&\,\,\le \frac{\partial_t\eta}{\eta}  - \sigma^{ij} \frac{\eta_{i;j}}{\eta} + \frac{\gamma'}{\gamma} \bigg(1 + \frac{1}{\varrho^2 W^2}\bigg)\big \langle \nabla^P \log \varrho, \nabla^P u \big \rangle + 2\frac{\gamma'}{\gamma}\frac{1}{\varrho^2 W}\frac{|\nabla^P u|}{W}\bigg|\frac{\nabla^P \eta}{\eta}\bigg| \\
	&- \frac{4|\nabla^P u|^2}{\varrho^2 W^4} \frac{\psi'}{\psi}\big \langle \nabla^P \log \varrho, \nabla^P u \big \rangle ^2 + \frac{4}{\varrho^2 W^4} \bigg \langle \frac{\nabla^P \eta}{\eta} + \frac{\gamma'}{\gamma} \nabla^P u, \nabla^P u \bigg \rangle \big \langle \nabla^P \log \varrho, \nabla^P u \big \rangle \\
	&-2\frac{\psi'}{\psi} \big[{\rm Ric}_g (\nabla^P u, \nabla^P u) - \nabla^2 \log \varrho (\nabla^P u, \nabla^P u)\big] + \frac{2}{\varrho^2 W^2}\frac{\psi'}{\psi}\nabla^2 \log \varrho (\nabla^P u, \nabla^P u) \\
	& - \bigg(1 + \frac{1}{\varrho^2 W^2}\bigg) \bigg \langle \nabla^P \log \varrho, \frac{\nabla^P \eta}{\eta} + \frac{\gamma'}{\gamma} \nabla^P u \bigg \rangle
	\end{align*}
	Discarding a non-positive term in the right hand side, we get
	\begin{align*}
	& \frac{\psi^2}{\psi'^2} \bigg(\bigg(\frac{\psi'}{\psi}\bigg)' - \frac{\psi'^2}{\psi^2} + \frac{3}{2W^2} \frac{\psi'}{\psi}
	\bigg) \frac{|\nabla^P u|^2}{ \varrho^2 W^2} 
	\bigg|\frac{\nabla^P \eta}{|\nabla^P u|\eta} +\frac{\gamma'}{\gamma}\frac{\nabla^P u}{|\nabla^P u|}\bigg|^2_g +\frac{\gamma''}{\gamma} \frac{|\nabla^P u |^2}{\varrho^2 W^2}
	\\
	&\,\,\le \frac{\partial_t\eta}{\eta}  - \sigma^{ij} \frac{\eta_{i;j}}{\eta} + 2\frac{\gamma'}{\gamma}\frac{1}{\varrho^2 W}\frac{|\nabla^P u|}{W}\bigg|\frac{\nabla^P \eta}{\eta}\bigg| + 4\frac{\psi'}{\psi}\frac{|\nabla^P \log \varrho|^2}{\varrho^2} \frac{|\nabla^P u|^4}{W^4}\\
	&+ 4 \frac{|\nabla^P \log \varrho |}{\varrho^2} \bigg( \bigg|\frac{\nabla^P \eta}{\eta}\bigg| \frac{|\nabla^P u|^2}{W^4} + \frac{\gamma'}{\gamma} \frac{|\nabla^P u|^3}{W^4}\bigg)   
    \\
    &-2\frac{\psi'}{\psi} \big[ {\rm Ric}_g (\nabla^P u, \nabla^P u) - \nabla^2 \log \varrho (\nabla^P u, \nabla^P u) \big] \\ 
	&+ 2 \frac{\psi'}{\psi}\frac{| \nabla^2\log \varrho|}{\varrho^2} \frac{| \nabla^P u |^2}{W^2} + \big|\nabla^P \log \varrho\big| \bigg|\frac{\nabla^P \eta}{\eta}\bigg| \bigg(1 + \frac{1}{\varrho^2 W^2}\bigg).
	\end{align*}
	If $|\nabla^P u(x_0, t_0)| \leq \alpha$ for some $\alpha > 0,$ we take $\psi(s) = s \quad \mbox{and}  \quad \eta, \gamma$ such that 
	$$  \eta(x, t) \leq \beta < \infty  \quad \mbox{and} \quad   1 + \min_{\overline{B_{R}}(o)}\varrho \geq \gamma(x,t) \geq 1 \qquad \mbox{in} \quad B_R(o)\times [0,T] $$
	and we obtain
	$$|\nabla u(x,t)| \, \leq \, \frac{1}{\beta} \, \chi(x,t) \, \leq  \, \frac{1}{\beta}\chi(x_0, t_0) \, \leq \, (1 + \inf_{\overline{B_{R}}(o)} \varrho) \alpha.$$
	Then, we supose that $|\nabla^P u(x_0,t_0)|^2 > 1$  and following \cite{Wang:98}, we set
	\begin{equation}
	\psi(\tau) = \log \tau,
	\end{equation}
	where $\tau=|\nabla^P u|^2$. We have
	\begin{align*}
	&\frac{|\nabla^P u|^2}{W^2}\frac{\psi^2}{\psi'^2} \bigg(\bigg(\frac{\psi'}{\psi}\bigg)' - \frac{\psi'^2}{\psi^2}+\frac{3}{2}\frac{1}{W^2}\frac{\psi'}{\psi}
	\bigg)
	= \frac{\tau}{W^2}\bigg( \log \tau \frac{\frac{1}{2}\tau-\varrho^{-2}}{\tau + \varrho^{-2}} -2 \bigg) .
	\end{align*}
	Now we consider $k>0$ be a constant such that $\tau > e^k$ and  we fix a  constant 
	\[
	\max \bigg\{ \frac{3}{4}, \ \frac{\varrho^2e^k }{1 + \varrho^2e^k}\bigg \} < \beta <1.
	\]
	We can  suppose that
	\begin{equation}\label{estnabla1}
	\frac{\tau}{W^2} =\frac{|\nabla^P u|^2}{W^2} \ge \beta.
	\end{equation}
	
	We also consider
	$\frac{1}{\varrho^2} \frac{\beta}{1 - \beta} =: e^{\delta'}, \ \delta = \frac{3}{2}\beta -1$ and $\mu := 2 \beta \frac{\delta \delta' - 2}{\delta'},$ and  we note that  
	$$\frac{1}{8} < \delta < \frac{1}{2}, \qquad   k < \delta' \leq \log \tau \qquad \mu > \frac{3}{16}\biggl(\frac{k - 16}{k}\biggr) > 0,$$
	if $k>16.$
	We get

	\begin{align*}
	&\mu \log |\nabla^P u| \frac{1}{\varrho^2} \bigg|\frac{\nabla^P \eta}{|\nabla^P u|\eta}
	+ \frac{\gamma'}{\gamma}\frac{\nabla^P u}{|\nabla^P u|}\bigg|^2_g  
	+\frac{\gamma''}{\gamma} \frac{|\nabla^P u|^2}{\varrho^2 W^2} \\ 
	&+ 2\frac{\psi'}{\psi} |\nabla^P u|^2\bigg( \,\Ric_g \bigg(\frac{\nabla^P u}{|\nabla^P u|}, \frac{\nabla^P u}{|\nabla^P u|}\bigg) - \nabla^2 \log \varrho \bigg(\frac{\nabla^P u}{|\nabla^P u|}, \frac{\nabla^P u}{|\nabla^P u|}\bigg) \bigg)\\
	&\,\, 
	\le \frac{1}{\eta}\big(\partial_t - \Delta \big) \eta + 2 \sqrt{1 - \beta} \frac{1}{\varrho}\frac{\gamma'}{\gamma} \bigg|\frac{\nabla^P \eta}{\eta}\bigg| + \frac{4}{\delta'}(1 - \beta)| \nabla^P \log \varrho |^2\\ 
	&+ \bigg|\frac{\nabla^P \eta}{\eta}\bigg|(6 - 5 \beta) | \nabla^P \log \varrho | + 4 \frac{\gamma'}{\gamma}\frac{1}{\varrho} \sqrt{1 - \beta}|\nabla^P \log \varrho| + \frac{2}{\delta'}(1 - \beta) | \nabla^2 \log \varrho|.
	\end{align*}
    We choose $\eta$ in $\mathcal{C}_{R,T}$ as
	\begin{equation}
	\eta =  (\zeta(R) - \zeta(r) - t)^2
	\end{equation}
	 where $r=d(o, \,\cdot).$\\
	 It follows from Proposition \ref{prop-par-s} that 
	\begin{eqnarray*}
		\frac{1}{\eta} \big(\partial_t - \Delta \big) \eta   
		 &=& -2 \frac{\sqrt{\eta}}{\eta}\big(\partial_t - \Delta \big)\zeta -2\frac{1}{\eta} |\nabla \zeta|^2  - 2\frac{\sqrt{\eta}}{\eta}\\
		& \le& \frac{2}{\sqrt{\eta}} \bigg(n \xi'(r)+ \varrho^2 |\nabla s|^2 \xi(r) \bigg(\big \langle \bar{\nabla}\log \varrho, \nabla r \big \rangle - \frac{\xi'(r)}{\xi(r)}\bigg)\bigg) \\ 
		&-& \frac{2 \xi(r)^2}{\eta} | \nabla r |^2 - 2\frac{\sqrt{\eta}}{\eta}.	
	\end{eqnarray*} \\
	Since $\big \langle \bar{\nabla} \log \varrho, \nabla r \big \rangle \leq \bigg| \frac{\partial_r \varrho}{\varrho} \bigg| \leq \frac{\xi'(r)}{\xi (r)}$ implies in $$\varrho^2 |\nabla s |^2 \xi(r)\bigg(\big \langle \bar{\nabla}\log \varrho, \nabla r \big \rangle - \frac{\xi'(r)}{\xi(r)}\bigg)\bigg) \leq 0,$$
	we obtain $$\frac{1}{\eta} \big(\partial_t - \Delta \big) \eta \leq 2n \frac{\xi'(r)}{\sqrt{\eta}}. $$
    Furthermore, $\bigg|\frac{\nabla \eta}{\eta}\bigg| \leq 2\frac{\xi(r)}{\sqrt{\eta}}.$
	
	Finally, we set as in \cite{Wang:98} 
	\begin{equation*}
	\gamma (u) = 1 +\frac{1}{M}( \min_{\overline{B}_{R}(o)} \varrho)u
	\end{equation*}
	where $M= {\sup}_{B_{R}(o)\times [0,T]} u>0$. Then $\gamma''=0$ and
	hence
	\begin{align*}
	\label{ineq-almost}
	\mu \log |\nabla u|\frac{1}{\varrho^2} &\bigg|\frac{\nabla \eta}{|\nabla u|\eta}
	+ \frac{\gamma'}{\gamma}\frac{\nabla u}{|\nabla u|}\bigg|^2_g  
	- 2L\frac{\psi'}{\psi} |\nabla u|^2  \le 
	\frac{1}{M \sqrt{\eta}}\bigg[2n M \xi'(r) + 4 \sqrt{1 - \beta} \xi(r)\\
	&+ \frac{2}{\delta'} (1 - \beta)\bigg(2| \nabla \log \varrho|^2 +  |\nabla^2 \log \varrho| \bigg)M \sqrt{\eta} + 2M \xi(r) (6 - 5 \beta)| \nabla \log \varrho| \\ 
	&+ 4 \sqrt{\eta} \sqrt{1 - \beta} \big| \nabla \log \varrho \big| \bigg]
	\end{align*}
	
	where  $L \ge 0$ is a constant such that
	\begin{equation}\label{Ldef}
	{\rm Ric}_g - \nabla^2 \log \varrho \ge -L g
	\end{equation}
	in $B_{R}(o)$. 
	Using that $ \frac{1}{\log t} \leq \frac{1}{\delta'}$ and $\eta \leq \zeta(R)^2$ we obtain

	\begin{align*}
	\mu \log |\nabla u| &\frac{1}{\varrho^2} \bigg|\frac{\nabla \eta}{|\nabla u|\eta}
	+ \frac{\gamma'}{\gamma}\frac{\nabla u}{|\nabla u|}\bigg|^2_g    \le   \frac{\zeta(R)}{M \eta} \bigg[\frac{2\zeta(R)ML}{\delta'} + 2nM \xi'(r) + 4 \sqrt{1 - \beta} \xi(r) \\ 
	&+ \frac{2}{\delta'} M\zeta(R) \big(1 - \beta\big) \bigg(2 | \nabla\log \varrho |^2 + | \nabla^2 \log \varrho| \bigg)  \\ 
	& + 2M \xi(r)(6 - 5 \beta)| \nabla \log \varrho| + 4\zeta(R) \sqrt{1 - \beta} | \nabla \log \varrho |  \bigg].
	\end{align*}
	
	We can to rewrite the inequality above as 
	
	\begin{align*}
	\eta \log |\nabla u| & \bigg|\frac{\nabla \eta}{|\nabla u|\eta}
	+ \frac{\gamma'}{\gamma}\frac{\nabla u}{|\nabla u|}\bigg|^2_g    \le   \frac{\zeta(R) \varrho^2}{\mu M} \bigg\{ \frac{2}{\delta'}LM \zeta(R) + 2nM \xi'(r) \\ 
	&  +  4 \sqrt{1-\beta}\xi(r)+ \frac{2}{\delta'}(1 - \beta)M \zeta(R)  \bigg(2 | \nabla\log \varrho |^2 + | \nabla^2 \log \varrho| \bigg) \\ 
	& +2M (6 - 5\beta)\xi(r)| \nabla \log \varrho| + 4 \zeta(R) \sqrt{1 - \beta} | \nabla \log \varrho |  \bigg\}.
	\end{align*}		
	We consider first the case
	\[
	\left| \frac{\nabla \eta}{|\nabla u| \eta} \right| \le \frac{\gamma'}{4\gamma}.
	\]
	Then we have
	\begin{equation*}
    \begin{split}
	\eta \log |\nabla u| &\le  {\frac{2M^2 \gamma^2}{\min \varrho^2} \frac{ \varrho^2 \zeta(R)}{\mu M}  \bigg\{ \frac{2}{\delta'}\zeta(R) L  M + 2nM \xi'(r)} 
    \\
    &+ 4\sqrt{1-\beta}\left(\xi(r)  +  \zeta(R) | \nabla \log \varrho |\right)
    \\
	&+ \frac{2}{\delta'} (1 - \beta)M \zeta(R)  \bigg(2 | \nabla \log \varrho |^2 + | \nabla^2 \log \varrho| \bigg)
    \\
    & + 2 M (6 - 5\beta)\xi(r)| \nabla \log \varrho|   \bigg\}\\
	&\le 2 \bigg(\frac{1 + \min \varrho}{\min \varrho}\bigg)^2  M\zeta^2(R) \frac{\varrho^2}{\mu} \bigg\{\frac{2}{\delta'}LM + 2nM \frac{\xi'(r)}{\zeta(R)} \\
	&+ 4 \sqrt{1-\beta}\frac{\xi(r)}{\zeta(R)} + \frac{2}{\delta'} (1-\beta) M\biggl(2|\nabla\log\varrho|^2 + |\nabla^2 \log \varrho|\biggr) 
    \\
    &+ 2\biggl((6 - 5\beta )M\frac{\xi(r)}{\zeta(R)}+ 2\sqrt{1-\beta}\biggr)|\nabla \log \varrho|\bigg\}
    \end{split}
	\end{equation*}
	If $\tau > e^k$ with $k = \max \biggl\{ML, \sup_{B_{R}(o)}\biggl(2 |\nabla \log \varrho|^2 + |\nabla^2\log\varrho|\biggr)\biggr\},$ then $\delta' > k$ implies 
	$$\frac{2}{\delta'} (1-\beta) M\biggl(2|\nabla\log\varrho|^2 + |\nabla^2 \log \varrho|\biggr) \leq 2(1-\beta) < \frac{1}{2}  \qquad \mbox{and} \qquad  \frac{2}{\delta'} ML \leq 2.$$ 
	Thus
	$$ \eta \log|\nabla u| \leq 4 \biggl(\frac{1 + \min \varrho}{\min \varrho}\biggr)^2 M \zeta^2(R) C_0, $$
	where
	\begin{equation} 
        \small {C_0 = \sup_{B_{R}(o)} \frac{\varrho^2}{\mu}\biggl\{ \frac{5}{4} + nM\frac{\xi'(r)}{\zeta(R)} + 2\sqrt{1 - \beta}\frac{\xi(r)}{\zeta(R)} + \biggl(M(6 - 5\beta) \frac{\xi(r)}{\zeta(R)} + 2 \sqrt{1 - \beta}\biggr) \frac{\varrho'}{\varrho}   \biggr\}}
    \end{equation}
	On the other hand, when 
	\[
	\frac{\gamma'}{4\gamma} \le \left| \frac{\nabla \eta}{|\nabla u| \eta} \right|
	\]
	we have
	\begin{equation*}
	\eta |\nabla u| \le \frac{8\gamma }{\gamma^\prime} \sqrt \eta | \nabla \zeta |  = \frac{8\gamma M \sqrt \eta \xi(r)}{\min_{\overline{B_R}(o)} \varrho}.
	\end{equation*}
	what implies that
	\begin{equation*}
	\eta \log|\nabla u| \le \frac{8\gamma M \sqrt{\eta} \xi(r)}{\min_{\overline{B_R}(o)} \varrho} \le  8\bigg(\frac{ 1 + \min_{\bar{B_R}(o)}\varrho }{\min_{\bar{B}(o, R)} \varrho}\bigg)M\zeta^2(R) \sup_{B_{R}(o)}\frac{\xi(r)}{\zeta(R)}.
	\end{equation*}
	
	Hence at $(x_0, t_0)$

	\begin{align*}
	\eta \log | \nabla u| \leq & 4 \bigg(\frac{ 1 + \min_{\bar{B_R}(o)}\varrho}{\min_{\bar{B_R}(o)} \varrho}\bigg)^2 M \zeta^2(R) \max \biggl\{ 2\sup_{B_R(o)}\frac{ \xi(r)}{\zeta(R)} , C_0\biggr\}.
	\end{align*}

	Since $\eta(x, t) > \frac{1}{16}\zeta^2(R)$ and $\gamma(x, t)\ge 1$ for $(x,t) \in B_{R'}(o) \times [0,T_{R'}]$ we conclude that 
	\begin{equation}
	\begin{split}
	\label{intgradest}
	\log |\nabla^P u(x,t)| &\le \frac{16 }{\zeta^2(R)}\eta(x, t)\gamma(x, t)\log |\nabla^P u(x, t)| 
	\\
    &\le \frac{16}{\zeta^2(R)}\gamma(x_0, t_0)\eta(x_0, t_0)\log|\nabla^P u(x_0, t_0)|
	\\
	&\le 
	\frac{16}{\zeta^2(R)} \bigg(1 + \min_{\overline{B_R}(o)}\varrho \bigg)  4\bigg(\frac{ 1 + \min_{\bar{B_R}(o)}\varrho }{\min_{\bar{B}(o, R)} \varrho}\bigg)^2M \zeta^2(R) \Xi\nonumber \\
	& = 64\frac{ \left(1 + \min_{\overline{B_R}(o)}\varrho\right)^2 }{\min_{\overline{B_R}(o)} \varrho}M \Xi
	\end{split}
	\end{equation}
	unless $|\nabla u(x_0, t_0)|\le 1,$ where
    \begin{equation}
        \Xi = \max \left\{2\sup_{B_{R}(o)}\frac{\xi(r)}{\zeta(R)} , C_0 \right\}
    \end{equation}
	
	Note that we might deal with $H_{2+\alpha}$ functions using a standard approximation argument. Moreover, we can drop the assumption that $u>0$ merely translating $u$ upwards by $M$.
	  
\end{proof}

\begin{corollary}\label{unif-grad-est}
	If $0< R'< R_1 < R_2$ such that  $\zeta(r) < \frac{1}{4}\zeta(R_1) < \frac{1}{4}\zeta(R_2) $ for all $r\leq R'$ and $u$ be a solution of  \eqref{R-approximate-problem} defined in  $B_{R_2}(o)\times [0, T]$ for  $T>0,$ then
	for $(x,t)$ in $B_{R'}(o)\times[0,T_{R'}]$
	either
	\begin{equation}
	\begin{split}
	\label{grad-est-0}
	|\nabla^P u(x,t)| \le \exp \left(128 \frac{ \left(1 + \min_{\overline{B}_{R_1}(o)}\varrho\right)^2 }{\min_{\overline{B}_{R_1}(o)} \varrho}M \sup_{B_{R_1}(o)}\frac{\xi(r)}{\zeta(R_1)} \right)
	\end{split}
	\end{equation}
	or
	\begin{equation}
	\begin{split}
	|\nabla^P u(x,t)| \le \exp \left(64 \frac{ \left(1 + \min_{\overline{B}_{R_1}(o)}\varrho\right)^2 }{\min_{\overline{B}_{R_1}(o)} \varrho}M C_0 \right), \nonumber
	\end{split}
	\end{equation}
	
	where $M = \sup_{B_{R_1}(p)\times [0, T]} u, $ \,  
	
	\begin{equation}
        \small{C_0 = \sup_{B_{R_1}(o)} \frac{\varrho^2}{\mu} \biggl\{\frac{5}{4} + nM\frac{\xi'(r)}{\zeta(R_1)} + 2\sqrt{1-\beta}\frac{\xi(r)}{\zeta(R_1)} 
	    + \biggl(M(6 - 5\beta)\frac{\xi(r)}{\zeta(R_1)}+ 2\sqrt{1-\beta}\biggr) \frac{\varrho'}{\varrho}\biggr\}}
    \end{equation}
    and
	$T_{R'}= \frac{1}{2}\zeta(R_1).$
\end{corollary}

\begin{proof}
	In fact, if we choose $$\psi (|\nabla^Pu|^2) = \log (|\nabla^Pu|^2), \quad \gamma (u) = 1 +\frac{1}{M}( \min_{\overline{B}_{R_1}(o)} \varrho)u,$$
with  $M= {\sup}_{B_{R_1}(o)\times [0,T]} u$ and  we define $\eta = (\zeta(R_1) - \zeta(r)- t)^2$ in $$\mathcal{C}_{R_1,T} = \big\{ \Psi(x,t) ; \, \, \zeta(r(\Psi(x,t))) + t < \zeta(R_1), \, \, x \in B_{R_1}(o), \, \, t \in [0,T]\big\}$$ we have the estimates announced.
\end{proof}

\subsection{Curvature estimates}

Given $R>0$ and $T>0$ we want to estimate $|\nabla^m|A||$ for $m\geq0$ in the parabolic cylinder $B_{R'}(o)\times [0,T_{R'}],$ where $R' \in (0,R)$ is such that $\zeta(r) < \frac{1}{4}\zeta(R)$ for all $r \leq R'$ and $T_{R'} = \frac{1}{2}\zeta(R).$ For this, we will proceed as Ecker-Huisken in \cite{EH91} studying the evolution of  the function 
\begin{equation}
f=\psi(W^2)|A|^2,
\end{equation}
where 
\begin{equation}
\label{psi-def}
\psi(W^2)=\frac{W^2}{\gamma-\delta W^2}
\end{equation}
with 
\[
\gamma =  {\inf}_{B_{R}(o)}\frac{1}{ \varrho^2} \qquad \mbox{and} \qquad  \delta = \frac{1}{2}\frac{\gamma}{{\sup}_{B_{R}(o)\times [0, T] } W^2 }.
\]

\noindent Initially, we need to deduce evolution equations for the second fundamental form and its squared norm, a variant of the classical Simons' formula.

\begin{lemma} \label{evol|A|}
	The squared norm $|A|^2$ of the second fundamental form of the graphs $\Sigma_t$, $t\in [0,T]$, evolve as
	\begin{equation}
	\label{par-simons}
	\begin{split}
	&  \frac{1}{2} (\partial_t-\Delta) |A|^2 + |\nabla A|^2  =|A|^4 +nHa^{ij}\bar{R}_{i00j}\\
	& \,\,+g^{k\ell}\left(\nabla_iL_{kj\ell}+\nabla_kL_{\ell ij}\right)a^{ij}+g^{k\ell} (a_{is}  \bar R^s_{kj\ell} +  a_{sk}  \bar R^s_{\ell ij}) a^{ij}
	\end{split}
	\end{equation}
	where $L$ is the $(0,3)$-tensor  in $\Sigma_t$ defined by $L_{ijk}=\langle\bar{R}(\partial_i,\partial_j)N,\partial_k\rangle$. This expression is rewritten in terms of the ambient curvature tensor as
	\begin{equation}
	\label{simons4}
	\begin{split}
	& \frac{1}{2} (\partial_t-\Delta) |A|^2 + |\nabla A|^2  =|A|^4  + |A|^2 \overline{{\rm Ric}}(N,N)
	\\
	& \,\,+g^{k\ell}(\bar\nabla_{i} \bar R_{kj0\ell} + \bar\nabla_k \bar R_{\ell i0j})a^{ij}  + 2  g^{k\ell} (a_{is}  \bar R^s_{kj\ell} +  a_{sk}  \bar R^s_{\ell ij}) a^{ij}.
	\end{split}
	\end{equation}
\end{lemma}

\begin{proof} We have
	\begin{equation*}
	\partial_ta_{ij}= n \nabla_i \nabla_j H - nH a_{is} a^s_j+nH\bar{R}_{i00j}
	\end{equation*}
	Since
	\[
	\partial_tg^{ij}=2nH a^{ij}
	\]
	we have
	\begin{eqnarray}\nonumber
	& \frac{1}{2}\partial_t|A|^2 = g^{j\ell} a_{ij} a_{k\ell}  \partial_t g^{ik} + 
	g^{ik} g^{j\ell} a_{k\ell}\partial_t a_{ij} = 2nHa^{ik} a_{i}^\ell a_{k\ell}  \\
	& \,\, +  a^{ij}  (n \nabla_i \nabla_j H - nH a_{i\ell} a^\ell_j+nH\bar{R}_{i00j}).
	\end{eqnarray}
	We conclude that
	\begin{equation}
	\frac{1}{2}\partial_t|A|^2 = nHa^{ik} a_{i}^\ell a_{k\ell}  +    n a^{ij} \nabla_i \nabla_j H +nHa^{ij}\bar{R}_{i00j}.
	\end{equation}
	On the other hand
	\begin{equation*}
	\label{simons1}
    \begin{split}
	\Delta a_{ij}&=n\nabla_i\nabla_jH+nHa_i^sa_{sj}-a_{ij}|A|^2-g^{k\ell}\left(\nabla_iL_{kj\ell}+\nabla_kL_{\ell ij}\right)
    \\
    &+g^{k\ell}(\bar{R}^s_{ik\ell }a_{sj}+\bar{R}^s_{ikj}a_{\ell s})
    \end{split}
	\end{equation*}
	and
	\begin{equation}
	\label{simons2}
	\begin{split}
	& \frac{1}{2}\Delta |A|^2 - |\nabla A|^2 = a^{ij} \Delta a_{ij} =n a^{ij}\nabla_i\nabla_jH+nHa_i^sa_{sj} a^{ij}-|A|^4\\
	& \,\,-g^{k\ell}\left(\nabla_iL_{kj\ell}+\nabla_kL_{\ell ij}\right)a^{ij}+g^{k\ell}(\bar{R}^s_{ik\ell }a_{sj}+\bar{R}^s_{ikj}a_{\ell s}) a^{ij}
	\end{split}
	\end{equation}
	Therefore
	\begin{equation}
	\begin{split}
	& \frac{1}{2} (\partial_t-\Delta) |A|^2 + |\nabla A|^2  =|A|^4 +nHa^{ij}\bar{R}_{i00j}\\
	& \,\,+g^{k\ell}\left(\nabla_iL_{kj\ell}+\nabla_kL_{\ell ij}\right)a^{ij}-g^{k\ell}(\bar{R}^s_{ik\ell }a_{sj}+\bar{R}^s_{ikj}a_{\ell s}) a^{ij}
	\end{split}
	\end{equation}
	It is worth to point out that
	\begin{eqnarray*}
		& & \nabla_i L_{kj\ell} + \nabla_k L_{\ell ij} = \bar\nabla_{i} \bar R_{kj0\ell} + \bar\nabla_k \bar R_{\ell i0j} + a_{ik} \bar R_{0j0\ell} + a_{ij} \bar R_{k00\ell} + a_{is} \bar R^s_{kj\ell}\\
		& & \,\, + a_{k\ell} \bar R_{0i0j} + a_{ki} \bar R_{\ell00 j} + a_{ks} \bar R^s_{\ell ij}.
	\end{eqnarray*}
	Therefore
	\begin{equation}
    \begin{split}
		g^{k\ell}(\nabla_i L_{kj\ell} + \nabla_k L_{\ell ij})a^{ij} &= g^{k\ell}(\bar\nabla_{i} \bar R_{kj0\ell} + \bar\nabla_k \bar R_{\ell i0j})a^{ij} - a_{i}^\ell a^{ij} \bar R_{j00\ell} 
        \\
        &+ |A|^2 \overline{{\rm Ric}}(N,N) + a_{is} a^{ij} g^{k\ell} \bar R^s_{kj\ell}- nH a^{ij} \bar R_{i00j}
        \\
        &+ a_{i}^\ell a^{ij} \bar R_{\ell00 j} + g^{k\ell} a^{ij} a_{sk}  \bar R^s_{\ell ij}.
    \end{split}
	\end{equation}
	Cancelling and grouping some terms one gets
	\begin{eqnarray*}
		& & g^{k\ell}(\nabla_i L_{kj\ell} + \nabla_k L_{\ell ij})a^{ij} = g^{k\ell}(\bar\nabla_{i} \bar R_{kj0\ell} + \bar\nabla_k \bar R_{\ell i0j})a^{ij}  + |A|^2 \overline{{\rm Ric}}(N,N) \\
		& & \,\,  + a_{is} a^{ij} g^{k\ell} \bar R^s_{kj\ell}- nH a^{ij} \bar R_{i00j}  + g^{k\ell} a^{ij} a_{sk} \bar R^s_{\ell ij}.
	\end{eqnarray*}

	Since
	\[
	-g^{k\ell}(\bar{R}^s_{ik\ell }a_{sj}+\bar{R}^s_{ikj}a_{\ell s}) a^{ij} =  a^{ij} g^{k\ell} (a_{is}  \bar R^s_{kj\ell} +  a_{sk}  \bar R^s_{\ell ij})
	\]
	we conclude that
	\begin{equation}
	\label{simons3}
	\begin{split}
	& \frac{1}{2} (\partial_t-\Delta) |A|^2 + |\nabla A|^2  =|A|^4  + |A|^2 \overline{{\rm Ric}}(N,N)
	\\
	& \,\,+g^{k\ell}(\bar\nabla_{i} \bar R_{kj0\ell} + \bar\nabla_k \bar R_{\ell i0j})a^{ij}  + 2  g^{k\ell} (a_{is}  \bar R^s_{kj\ell} +  a_{sk}  \bar R^s_{\ell ij}) a^{ij}.
	\end{split}
	\end{equation}
	This finishes the proof. 
\end{proof}

\vspace{1cm}

\noindent Now, let us consider the set 
$$\mathcal{C}_{R,T} = \left\{ \Psi(x,t) ; \, \, \zeta(r(\Psi(x,t))) + t < \zeta(R), \, \, x \in B_R(p), \, \, t \in [0,T]\right\}.$$   We will prove the following estimate.

\begin{proposition}\label{curvature-estimate}
	Let $u$ be a solution of (\ref{mcf1})-(\ref{mod-flow})  defined in $B_R(o) \times [0,T].$ If there exists $L_1 > 0$ such that $\overline{\rm Ric} \geq - L_1 \overline{g},$ then 
	
   	\begin{equation}\label{curv-est}
   	\sup_{B_{R'}(o) \times [0, T_{R'}]}|A| \leq \frac{4}{\sqrt{\delta}}  \biggl[ 1 +  L_1 + \widetilde{C} + C+  \frac{E_R}{\zeta^2(R)} + \frac{1}{2T} \bigg]^\frac{1}{2} 
   	\end{equation}
   	where $E_R, C \, \mbox{and} \, \widetilde{C}$ are non-negative constants depending on $\xi, \varrho$ and its derivatives.
   	Moreover, for $m \geq 1$
   	
   	\begin{equation}\label{curv-est-order-m}
   	 \sup_{B_{R'}(o)\times [0,T_{R'}]} |\nabla^m A| \leq C_m \bigg( \sup_{B_{R}(o)\times[0,T]}W^2, \xi(R),\zeta(R), L_1, C, \widetilde{C}, E_R \bigg).
   	\end{equation}
\end{proposition}

\noindent In order to prove this proposition we will study the evolution of the function $f= \psi(W^2)|A|^2.$

\begin{lemma} If there exists constant $L_1 > 0$ such that $\overline{\Ric} \geq -L_1\overline{g}$, then
	\begin{equation}
	\small{\left(\partial_t-\Delta\right)f\leq - 2 \delta f^2 + 2(L_1 + \tilde{C})f + 2C\sqrt{\psi}\sqrt{f} - \frac{2\gamma}{W^3} \psi \langle \nabla W , \nabla f \rangle - 2 \delta \psi' |\nabla W|^2 f}
	\end{equation}
	where $C$ and $\widetilde C$ are non-negative constants depending on $\varrho$ and its derivatives.
\end{lemma}

\begin{proof}
	We have
	
	\begin{equation}
    \begin{split}
		(\partial_t - \Delta)f =& 2|A|^2 \psi' W(\partial_t - \Delta)W + \psi (\partial_t - \Delta)|A|^2  \\ 
		-& 2|A|^2 (2 \psi'' W^2 + \psi')|\nabla W|^2- 2 \langle \nabla \psi, \nabla |A|^2 \rangle \\ 
		=& - 2|A|^2\psi' W\left(W(|A|^2 + \overline{\Ric}(N,N)) + 2W^{-1}|\nabla W|^2\right) \\ 
		+&  2\psi \left(|A|^2(|A|^2 + \overline{\Ric}(N,N)) - |\nabla A|^2 + \mathcal{R} \right) \\ 
		-& 2 \langle \nabla \psi, \nabla |A|^2 \rangle - 2|A|^2 (2 \psi'' W^2 + \psi')|\nabla W|^2 \\ 
		=& 2\left(\psi |A|^2 - |A|^2 \psi' W^2\right)\left(|A|^2 + \overline{\Ric}(N,N)\right) - 2\psi|\nabla A|^2 + 2\psi \mathcal{R} \\ 
		-& 2|A|^2\left(3\psi' + 2\psi''W^2\right)|\nabla W|^2 - 2 \langle \nabla \psi, \nabla|A|^2 \rangle
    \end{split}
	\end{equation}
	
	where in the second equality we used (\ref{evolW}) and (\ref{par-simons}) and we denote 
	$$\mathcal{R}:= g^{k\ell}\left(\nabla_iL_{kj\ell}+\nabla_kL_{\ell ij}\right)a^{ij}+g^{k\ell} (a_{is}  \bar R^s_{kj\ell} +  a_{sk}  \bar R^s_{\ell ij}) a^{ij} $$
	A lenghty calculation based in \cite{Borisenko-Miquel} shows that there exists non-negative constants $C$ and $\widetilde{C}$ depending on $\varrho$ and its derivatives such that $\mathcal{R} \leq C|A| + \widetilde{C}|A|^2$. Then,
	$$2 \psi \mathcal{R} \leq 2C \sqrt{\psi}\sqrt{f} + 2 \widetilde{C}f.$$
	The Kato's inequality implies to $$- 2\psi |\nabla A|^2 \leq -2 \psi |\nabla|A||^2.$$
	Moreover, since that $$\psi |A|^2 - |A|^2\psi' W^2 = \big(1 - \frac{\gamma}{\gamma - \delta W^2}\big)f = - \delta \psi f,$$
	we have 
    \begin{equation}
    \begin{split}
        2\left(\psi |A|^2 - |A|^2 \psi' W^2\right)\left(|A|^2 + \overline{\Ric}(N,N)\right) &= -2 \delta f^2 - 2\delta \psi f \overline{\Ric}(N,N) 
        \\
        &\leq -2\delta f^2 + 2\delta \psi L_1 f.
    \end{split}
    \end{equation}
	Therefore
	\begin{equation}
    \begin{split}
		(\partial_t - \Delta)f \leq& -2\delta f^2 + 2\delta \psi L_1 f -2\psi |\nabla|A||^2 + 2C\sqrt{\psi}\sqrt{f} + 2 \widetilde{C}f \\
		-& \big(6\psi' + 4\psi'' W^2\big)|A|^2 |\nabla W|^2 - 2\langle \nabla |A|^2, \nabla \psi \rangle.
    \end{split}
	\end{equation}
	We note that
	\begin{equation}
    \begin{split}
		- 2\langle \nabla |A|^2, \nabla \psi \rangle =& - \langle \nabla|A|^2, \nabla \psi \rangle  - \psi^{-1} \langle \nabla \psi, \psi \nabla |A|^2 \rangle \\
		=& - 4\psi' W|A| \langle \nabla W, \nabla |A| \rangle - \psi^{-1}\langle \nabla \psi, \nabla f \rangle + \psi^{-1}|A|^2 |\nabla \psi|^2 \\
		=& -  \psi^{-1}\langle \nabla \psi, \nabla f \rangle + 4\psi^{-1}|A|^2{\psi'}^2W^2|\nabla W|^2 
        \\
        -& 4 \psi' W |A| \langle \nabla W, \nabla |A| \rangle.
    \end{split}
	\end{equation}
	Using Young's inequality, we obtain 
	$$ 4 \psi' W |A| \langle \nabla W, \nabla |A| \rangle + 2\psi|\nabla|A||^2 + 2 \psi^{-1} {\psi'}^2 W^2 |A|^2 |\nabla W|^2 \geq 0.$$
	Therefore
	$$-2 \langle \nabla |A|^2, \nabla \psi \rangle \leq - \psi^{-1} \langle \nabla \psi, \nabla f \rangle + 6\psi^{-1}|A|^2{\psi'}^2W^2|\nabla W|^2  + 2 \psi |\nabla|A||^2$$ 
	hence
	\begin{eqnarray}\label{evolf}
	(\partial_t - \Delta)f &\leq& -2\delta f^2 + 2(\delta \psi L_1 + \widetilde{C})f  + 2C\sqrt{\psi}\sqrt{f} - \psi^{-1}\langle \nabla \psi, \nabla f \rangle \\
	&\,& - \bigg(6 \psi' \bigg( 1 - \frac{\psi'}{\psi}W^2 \bigg) + 4 \psi'' W^2\bigg)|A|^2 |\nabla W|^2. \nonumber 
	\end{eqnarray}
	
	Since
	$$\psi - \psi' W^2 = \frac{W^2}{\gamma - \delta W^2}\bigg( 1 - \frac{\gamma}{\gamma - \delta W^2} \bigg) = - \frac{\delta W^2}{\gamma - \delta W^2} \psi =  - \delta \psi^2,$$  we have 
	\begin{equation}
    \begin{split}
    - \bigg(6 \psi' \bigg( 1 - \frac{\psi'}{\psi}W^2 \bigg) + 4 \psi'' W^2\bigg) &= - \bigg(6 \frac{\psi'}{\psi}(- \delta \psi^2) + \frac{8 \gamma \delta}{(\gamma - \delta W^2)^3}W^2 \bigg) 
    \\
    &= -2 \delta \psi' \psi.
    \end{split}
    \end{equation}
	Hence
	\begin{equation}\label{gradW}
	- \bigg(6 \psi' \bigg( 1 - \frac{\psi'}{\psi}W^2 \bigg) + 4 \psi'' W^2\bigg)|A|^2 |\nabla W|^2 = -2 \delta \psi' f |\nabla W|^2.
	\end{equation}
	Moreover
	\begin{equation}\label{grad-psi}
	\psi^{-1} \nabla \psi = 2 \psi' \psi^{-1}W \nabla W = 2 \frac{\gamma}{(\gamma - \delta W^2)^2} \frac{\gamma - \delta W^2}{W^2}W \nabla W = \frac{2 \gamma}{W^3}\psi \nabla W 
	\end{equation}
	and
	\begin{equation}\label{delta-psi}
	0 \le \delta \psi = \frac{\delta W^2}{\gamma - \delta W^2} \leq \frac{\frac{\gamma}{2}}{\gamma - \frac{\gamma}{2}} = 1.
	\end{equation}
	
	Therefore, using $(\ref{gradW}), (\ref{grad-psi})$ and $(\ref{delta-psi})$ we rewrite \ref{evolf} as 
	\begin{eqnarray*}
		(\partial_t - \Delta)f &\leq& -2 \delta f^2 + 2(L_1 + \widetilde{C})f + 2C\sqrt{\psi}\sqrt{f} - 2 \delta \psi' |\nabla W|^2 f \\
		&\,& -  \frac{2 \gamma}{W^3}\psi \langle \nabla W, \nabla f \rangle.
	\end{eqnarray*}
	
\end{proof}

\noindent Now, we can prove the Proposition $\ref{curvature-estimate}.$

\begin{proof}

Let $\eta$ be a smooth function defined in $\mathcal{C}_{R,T}$ by 
$$\eta(\Psi(x,t)) = \big( \zeta(R) - \zeta(r(\Psi(x,t))) - t \big)^2 .$$

\noindent We have by Proposition \ref{prop-par-s} and (\ref{cylinder-0}) that
\begin{equation}
\begin{split}
	(\partial_t - \Delta)\eta &= -2 \sqrt{\eta} \big(\partial_t - \Delta \big)\zeta - 2 \sqrt{\eta} - 2|\nabla \zeta|^2 \\
	& \leq 2\sqrt{\eta}\bigg (n\xi'(r) + \varrho ^2|\nabla s|^2 \xi(r) \bigg(\langle \overline{\nabla} \log \varrho, \nabla r \rangle - \frac{\xi'(r)}{\xi(r)}\bigg) \bigg) 
    \\
    &- 2\sqrt{\eta} - 2|\nabla \zeta|^2\\
	& \leq 2n \xi'(r)\sqrt{\eta}.
\end{split}
\end{equation}
Then
\begin{equation}
\begin{split}
	(\partial_t - \Delta)(\eta f) =& \eta (\partial_t - \Delta)f + f (\partial_t - \Delta)\eta - 2 \langle \nabla f, \nabla \eta \rangle \\
	\leq&  -2\delta \eta f^2 + 2(L_1 + \widetilde{C}) \eta f + 2C\sqrt{\psi} \sqrt{f}\eta - \frac{2\gamma}{W^3}\psi \eta \langle \nabla W, \nabla f \rangle \\
	 -& 2 \delta \psi' |\nabla W|^2 \eta f + 2n \xi'(r) \sqrt{\eta} f - 2\langle \nabla \eta, \nabla f \rangle.
\end{split}
\end{equation}

We observe that 
\begin{align*}
-2 \langle \nabla \eta, \nabla f \rangle &= -2 \eta^{-1} \langle \nabla \eta, \eta \nabla f \rangle = -2 \eta^{-1} \langle \nabla \eta,  \nabla (\eta f) \rangle + 2 \eta^{-1}  |\nabla \eta|^2f \\
&=  -2 \eta^{-1} \langle \nabla \eta,  \nabla (\eta f) \rangle + 8\xi^2(r) f
\end{align*}
and

\begin{equation*}
\begin{split}
	- 2 \gamma W^{-3}\psi \eta \langle \nabla W, \nabla f \rangle =& -2 \gamma W^{-3} \psi  \langle \nabla W, \eta \nabla f \rangle \\
	=& -2 \gamma W^{-3} \psi  \langle \nabla W, \nabla(\eta f) \rangle + 2 \gamma W^{-3} \psi f \langle \nabla W, \nabla \eta \rangle \\
	\leq& -2 \gamma W^{-3} \psi  \langle \nabla W, \nabla (\eta f) \rangle + 2\delta \psi' \eta f|\nabla W|^2 
    \\
    &+ 2\frac{\gamma }{ \delta W^2}\xi^2(r) f,
\end{split}
\end{equation*}
in which we used Young's inequality.
Therefore
\begin{equation}\label{evol-eta-f}
\begin{split}
	(\partial_t - \Delta)(\eta f) \leq&  -2\delta \eta f^2 + 2(L_1 + \widetilde{C}) \eta f + 2C\sqrt{\psi} \eta \sqrt{f} 
    \\
    -&2 \gamma W^{-3}\psi \langle \nabla W, \nabla (\eta f)\rangle \\
	+& 2 \delta \psi' \eta f |\nabla W|^2 + \frac{2 \gamma }{ \delta W^2} \xi^2(r)f - 2\delta \psi' |\nabla W|^2 \eta f \\
	+&2n \xi'(r)\sqrt{\eta}f -2 \eta^{-1} \langle \nabla \eta, \nabla(\eta f) \rangle + 8 \xi^2(r) f \\
	=&- 2\delta \eta f^2 + 2(L_1 + \widetilde{C}) \eta f + 2C\sqrt{\psi} \eta \sqrt{f} 
    \\
    -& 2 \bigg \langle \frac{\gamma \psi}{W^3} \nabla W + \frac{\nabla \eta}{\eta}, \nabla(\eta f) \bigg \rangle \\
	+& 2 \biggl(\frac{\gamma}{\delta W^2}\xi^2(r) +n \xi'(r)\sqrt{\eta} + 4\xi^2(r)  \biggr) f.
\end{split}
\end{equation}

\noindent It follows from  $\frac{\gamma}{W^2} \leq 1$ and $ \sqrt{\eta} \leq \zeta (R)$  in $\mathcal{C}_{R,T}$ that we have 
 
 \begin{equation}
    \begin{split}
    \biggl(\frac{\gamma}{\delta W^2} + 4 \biggr)\xi^2(r) + n \sqrt{\eta} \xi'(r) &\leq \biggl(\frac{1}{\delta} + 4 \biggr) \sup_{B_{R}(o) \times [0, T]}\xi^2(r) 
    \\
    &+ n \zeta(R) \sup_{B_{R}(o) \times [0, T]}|\xi'(r)| := E_R.
    \end{split}
 \end{equation}

\noindent Therefore

\begin{equation}
\begin{split}
    (\partial_t - \Delta)(\eta f) &\leq - 2\delta \eta f^2 + 2(L_1 + \widetilde{C})\eta f + \frac{2C}{\sqrt{\delta}}\eta \sqrt{f} + 2E_R f \\
	&- 2 \bigg \langle \frac{\gamma \psi}{W^3} \nabla W + \frac{\nabla \eta}{\eta}, \nabla(\eta f) \bigg \rangle.
\end{split}
\end{equation}
Hence
\begin{eqnarray*}
	(\partial_t - \Delta)(\eta ft) &=& t(\partial_t - \Delta)(\eta f) + \eta f \\
	&\leq& - 2\delta \eta f^2t + 2(L_1 + \widetilde{C})\eta ft + \frac{2C}{\sqrt{\delta}}\eta \sqrt{f}t + 2 E_R ft \\
	&\,& - 2 \bigg \langle \frac{\gamma \psi}{W^3} \nabla W + \frac{\nabla \eta}{\eta}, \nabla(\eta ft) \bigg \rangle + \eta f.
\end{eqnarray*}
 
 \noindent Let $(x_0, t_0)$ be the point where the function $\eta ft$ attains a maximum value $M_R$ in  $\mathcal{C}_{R,T}. $ We can suppose that  $t_0 \neq 0$ and we note that 
 
 \[
 2\delta \eta f^2 t_0 \leq 2(L_1 + \widetilde{C})\eta ft_0 + \frac{2C}{\sqrt{\delta}}\eta \sqrt{f}t_0 + 2 E_R ft_0 + \eta f.  
 \]
 So multiplying by  $\eta t_0/2 \delta$ and grouping the terms we have
 
 \begin{align*}
  M^2_R &\leq \frac{L_1 + \widetilde{C}}{\delta}\eta t_0 M_R + \frac{C}{\delta^{\frac{3}{2}}} \eta^{\frac{3}{2}} t_0^{\frac{3}{2}} \sqrt{M_R} + \frac{E_R}{\delta} t_0 M_R + \frac{\eta}{2 \delta}M_R  \\
  &\leq \biggl[ \biggl(\frac{L_1 + \widetilde{C}}{\delta}\zeta^2(R) + \frac{E_R}{\delta}\biggr)T + \frac{\zeta^2(R)}{2 \delta}\biggr]M_R + \frac{C}{\delta^{\frac{3}{2}}}\zeta^3(R)T^{\frac{3}{2}}\sqrt{M_R}
  \end{align*}
  
  \noindent where in the last inequality we used that $t_0 \leq T$ and $\eta \leq \zeta^2(R)$ in $\mathcal{C}_{R,T}.$
 	
 \noindent Therefore
 \[
 \bigg(\frac{\sqrt{M_R}}{\zeta(R)}\bigg)^3 - \frac{1}{\delta}\biggl[\bigg(L_1 + \widetilde{C} + \frac{E_R}{\zeta^2(R)}\bigg)T + \frac{1}{2} \bigg] \frac{\sqrt{M_R}}{\zeta(R)} - \frac{C}{\delta^\frac{3}{2}}T^\frac{3}{2} \leq 0  
 \] 
 or yet
 
 \[ 
 \biggl[ \bigg(\frac{\sqrt{M_R}}{\zeta(R)}\bigg)^2 - \frac{1}{\delta} \biggl[ \bigg( L_1+ \widetilde{C} +  \frac{E_R}{\zeta^2(R)} \bigg) T + \frac{1}{2}\biggr] \biggr] \frac{\sqrt{M_R}}{\zeta(R)} - \frac{C}{\delta^\frac{3}{2}} T^\frac{3}{2}_R \leq 0 \]
 
 \noindent In this case, either 
 \[  
 \bigg(\frac{\sqrt{M_R}}{\zeta(R)}\bigg)^2 - \frac{1}{\delta} \biggl[ \bigg( L_1+ \widetilde{C} +  \frac{E_R}{\zeta^2(R)} \bigg) T + \frac{1}{2}\biggr]  \leq \frac{CT}{ \delta} 
  \]
 which leads us to  $  \frac{\sqrt{M_R}}{\zeta(R)} \leq  \frac{1}{\sqrt{\delta}}  \biggl[ \biggl( L_1 + \widetilde{C} + C + \frac{E_R}{\zeta^2(R)}\biggr)T + \frac{1}{2} \bigg]^\frac{1}{2},$ \\

 or
 
  \[
   \frac{CT}{\delta}\frac{\sqrt{M_R}}{\zeta(R)} \leq \frac{C}{\delta^\frac{3}{2}}T^\frac{3}{2} . 
   \]
 
 \noindent Thus 
 \[
  \frac{\sqrt{M_R}}{\zeta(R)} \leq \max \bigg\{  \frac{1}{\sqrt{\delta}}  \biggl[  \bigg( L_1 + \widetilde{C} + C + \frac{E_R}{\zeta(R)^2}  \bigg) T + \frac{1}{2} \biggr]^\frac{1}{2}, \, \frac{T^{\frac{1}{2}}}{\sqrt{\delta}} \bigg\} :=C_R
  \] 
  
   \noindent hence
 
 \[ 
 \frac{\sqrt{\eta f t }}{\zeta(R)}(x,t) \leq \frac{\sqrt{\eta f t }}{\zeta(R)}(x_0,t_0)  \leq C_R \qquad \forall \quad (x,t) \in B_{R'}(o) \times [0, T_{R'}].
 \] 
 That is 
 \[\bigg(1 - \frac{\zeta(r) + t}{\zeta(R)}\bigg)\sqrt{\psi}|A|t^\frac{1}{2} \leq C_R \quad \mbox{in} \quad B_{R'}(o) \times [0, T_{R'}]. \] \\
 Since $ \psi= \frac{W^2}{\gamma - \delta W^2} \geq  \frac{\varrho ^{-2}}{\gamma - \delta \varrho^{-2}} \geq \frac{1}{1 - \delta} \geq 1 $ and  $\zeta(r) + t < \frac{3}{4}\zeta(R)$ in $B_{R'}(o) \times [0, T_{R'}]$
  we have
 \[  |A| \sqrt{t} \leq 4 C_R  \quad
  \mbox{in} \quad B_{R'}(o) \times [0, T_R],\] 
  \noindent 	hence 
  	\[
  	\sup_{B_{R'}(o) \times [0, T_{R'}]}|A| \leq \frac{4}{\sqrt{T}}C_R \] 
 
 Therefore
 \begin{align*}
 \sup_{B_{R'}(o) \times [0, T_{R'}]}|A| &\leq \frac{4}{\sqrt{\delta}}\max \bigg\{  \biggl( L + \widetilde{C} + C + \frac{E_R}{\zeta(R)^2} + \frac{1}{2T} \biggr)^\frac{1}{2}, \, 1  \bigg\} \\
 &\leq   \frac{4}{\sqrt{\delta}}  \biggl(1 + L + \widetilde{C} + C + \frac{E_R}{\zeta(R)^2} + \frac{1}{2T} \biggr)^\frac{1}{2}
 \end{align*}
 
 For the estimate in $( \ref{curv-est-order-m})$, we proceed inductively as Ecker-Huisken in \cite{EH91} and Borisenko-Miquel in \cite{Borisenko-Miquel}. We suppose that for each $k =0, 1,\ldots, \ell-1$   there exists a constant $C_k $ such that
 \[
 |\nabla^kA|\leq C_k  
 \]
 where $C_k$ depends on the bounds of $|\nabla^m A|,$  on the tensors $\bar\nabla^m \bar R$ for $0\leq m\leq k-1$  and on the geometric data in $B_{R}(o)\times [0, T]$. 
 
 As in  \cite{EH91} and \cite{Borisenko-Miquel} we will use variants of the Simons' inequality for higher order covariant derivatives of $A$ which have the form
 \begin{equation}
 \frac{1}{2}(\partial_t- \Delta)|\nabla^\ell A|^2+|\nabla^{\ell+1}A|^2\le D_\ell (|\nabla^\ell A|^2+1)
 \end{equation}
 where the constant $D_\ell$ depends on the bounds of $|\nabla^kA|$  and on the tensors $\bar\nabla^k \bar R$ for $0\leq k\leq \ell-1$ in $B_R(o)\times [0, T]$.
 We consider the function
 \[
 h=|\nabla^\ell A|^2+\beta|\nabla^{\ell -1}A|^2
 \]
 where $\beta$ is a positive constant to be choosen later. Setting $\beta \ge 2D_\ell$ one obtains
 \begin{align*}
  \frac{1}{2}\partial_t  h   &\le  \frac{1}{2} \Delta |\nabla^\ell A|^2 - |\nabla^{\ell+1}A|^2 + D_\ell (|\nabla^\ell A|^2+1) \\
  & \, \, + \frac{1}{2} \beta \Delta |\nabla^{\ell -1}A|^2 - \beta|\nabla^{\ell}A|^2 + \beta  D_{\ell -1} (|\nabla^{\ell -1}A|^2+1)\\ 
 & \le  \frac{1}{2} \Delta h + (D_{\ell} - \beta)|\nabla^{\ell}A|^2 + \beta D_{\ell -1}|\nabla^{\ell -1}A|^2 + D_{\ell} + \beta D_{\ell -1} \\
 & \le \frac{1}{2}\Delta h   - \frac{\beta}{2} |\nabla^{\ell}A|^2 + \beta D_{\ell-1} |\nabla^{\ell-1} A|^2 + D_\ell + \beta D_{\ell-1}\\
 & \le \frac{1}{2}\Delta h   - \frac{\beta}{2}  h + \frac{\beta^2}{2} |\nabla^{\ell-1} A|^2 +  \beta D_{\ell-1} |\nabla^{\ell-1} A|^2 + D_\ell + \beta D_{\ell-1}
 \end{align*}
 Choosing $\beta \ge 2D_{\ell-1}$ we obtain
 \begin{equation}
 (\partial_t - \Delta) h \le  -\beta  h + \beta^2 \widetilde{C}_\ell + \widetilde D_{\ell},
 \end{equation}
 where $\widetilde C_\ell =  2|\nabla^{\ell-1} A|^2$ and $\widetilde D_\ell = 2D_\ell + 2\beta D_{\ell-1}$ .
 Again, we consider $\eta$ defined in $\mathcal{C}_{R,T}$ as $\eta(\psi(x,t)) = (\zeta(R) - \zeta(r(\psi(x,t))) - t)^2.$ Then, we have
 
 \begin{align*}
  \left(\partial_t-\Delta\right)\eta &=  - 2 \sqrt{\eta} \left(\partial_t-\Delta\right)\zeta - 2 |\nabla \zeta|^2 \\
  &\leq 2 \sup_{B_{R}(o) \times [0, T]}n\xi'(r) \sqrt{\eta}  - 2 |\nabla \zeta|^2
:=  2 C_R \sqrt{\eta}  - 2 |\nabla \zeta|^2 
\end{align*}
 
 and
 \begin{equation}
 	\begin{split}
 		& \left(\partial_t-\Delta\right)(\eta h)= h\left(\partial_t-\Delta\right)\eta + \eta \left(\partial_t-\Delta\right) h -2\langle\nabla\eta ,\nabla h\rangle\\
 		&\,\, \leq 2C_R \sqrt{\eta}h - 2 h |\nabla\zeta|^2 + (-\beta  h + \beta^2 \widetilde{C}_\ell + \widetilde D_{\ell})\eta -2\left\langle\eta^{-1}\nabla\eta ,\nabla(\eta h)- h\nabla\eta \right\rangle.
 	\end{split}
 \end{equation}
 Therefore
 \begin{equation}
 	\begin{split}
 		\left(\partial_t-\Delta\right)(\eta h)+2\left\langle\eta^{-1}\nabla\eta,\nabla(\eta h)\right\rangle &\le (2C_R \sqrt{\eta} - 2  |\nabla\zeta|^2 + 2\eta^{-1}|\nabla \eta|^2) h
        \\
        &+ (-\beta  h + \beta^2 \widetilde{C}_\ell + \widetilde D_{\ell})\eta.
 	\end{split}
 \end{equation}
 It follows from $
 -2|\nabla \zeta|^2+ 2\eta^{-1}|\nabla \eta|^2 = 6 |\nabla \zeta|^2$ that
 
 \begin{equation}
 	\begin{split}
 		\left(\partial_t-\Delta\right)(\eta h)+2\left\langle\eta^{-1}\nabla\eta,\nabla(\eta h)\right\rangle &\le  (2C_R\zeta(R) + 6 |\nabla\zeta|^2  - \beta \eta) h 
        \\
        &+ (\beta^2 \widetilde{C}_\ell + \widetilde D_{\ell})\eta.
 	\end{split}
 \end{equation}
 We have at a maximum point of  $\eta h$  that
 \[
 (\beta \eta - 2C_R \zeta(R) -6\xi^2(r))h \leq (\beta^2 \widetilde{C}_\ell + \widetilde D_{\ell})\zeta^2(R)
 \]
 Since that $\eta \geq\zeta^2(R)/16$ in  $B_{R'} (p) \times [0, T_R]$ we have
 \[
 \left( \frac{\beta}{16} - \frac{2C_R}{\zeta(R)}-6\frac{\xi^2(R)}{\zeta^2(R)}\right) h \leq \beta^2 \widetilde{C}_\ell + \widetilde D_{\ell}
 \]
 Choosing 
 \[
 \beta\ge \max\left\{ 2 D_\ell, 2D_{\ell-1}, \frac{32}{\zeta^2(R)}\left(C_R \zeta(R) + 6 \xi^2(R) \right) \right\}
 \] 
 we get
 \[
 h\leq \frac{1}{6} \frac{\zeta^2(R)}{\xi^2(R)} \big(\beta^2 \widetilde{C}_\ell + \widetilde D_{\ell}\big).
 \]
 Thus
 \begin{equation}
 |\nabla^\ell A|^2\leq \frac{1}{6} \frac{\zeta^2(R)}{\xi^2(R)} \big(2\beta^2 |\nabla^{\ell-1}A|^2
 + \widetilde D_{\ell}\big) - \beta |\nabla^{\ell-1}A|^2.
 \end{equation}
 A suitable choice of a large enough $\beta$ yields the desired estimate in $(\ref{curv-est-order-m}).$
 \end{proof}
 
 \begin{corollary}\label{unif-curv-estimate}
 	Let  $L_1 > 0$ be a constant such that $\overline{\rm Ric} \geq - L_1 \overline{g}.$ If $0< R'< R_1 < R_2$  are such that  $\zeta(r) < \frac{1}{4}\zeta(R_1) < \frac{1}{4}\zeta(R_2) $ for all $r\leq R'$ and $u$ be a solution of  \eqref{R-approximate-problem} defined in  $B_{R_2}(o)\times [0, T]$ for  $T>0,$  then 
	
	\begin{equation}\label{unif-curv-est}
	\sup_{B_{R'}(o) \times [0, T_{R'}]}|A| \leq \frac{4}{\sqrt{\delta}}  \biggl[ 1 +  L_1 + \widetilde{C} + C+  \frac{E_{R_1}}{\zeta^2(R_1)} + \frac{1}{2T} \bigg]^\frac{1}{2} 
	\end{equation}
	where $E_{R_1}, C \, \mbox{and} \, \widetilde{C}$ are non-negative constants depending on $\xi, \varrho$ and its derivatives.
	Moreover, for $m \geq 1$
	
	\begin{equation}\label{unif-curv-est-order-m}
	\sup_{B_{R'}(o)\times [0,T_{R'}]} |\nabla^m A| \leq C_m \bigg( \sup_{B_{R_1}(o)\times[0,T]}W^2, \xi(R_1),\zeta(R_1), L_1, C, \widetilde{C}, E_{R_1} \bigg).
	\end{equation}
 \end{corollary}
\begin{proof}
	 In fact, if  we define $$h = |\nabla^\ell A|^2+\beta|\nabla^{\ell -1}A|^2 \qquad \mbox{in} \qquad B_{R_1}(o) \times [0,T]$$ and  $$\eta = (\zeta(R_1) - \zeta(r)- t)^2$$
     in $\mathcal{C}_{R_1,T} = \big\{ \Psi(x,t) ; \, \, \zeta(r(\Psi(x,t))) + t < \zeta(R_1), \, \, x \in B_{R_1}(o), \, \, t \in [0,T]\big\}$, we have then the estimates announced.
\end{proof}

\section{Existence of the Flow}\label{Existence}
In the following section, we prove the main theorem. The argument begins by restricting the PDE to a compact subset, where classical existence and regularity results apply. After establishing the necessary estimates and properties in this setting, we extend the result to the non-compact case through a limiting process that controls the solution’s behavior at infinity.
\subsection{Existence of the flow in compact case} 
 \begin{theorem}[Compact Case]\label{ball-case}	For $R>0,$ let $B_R = B(o,R) \subset P$ be a geodesic ball and $\Psi_0 : B_R \rightarrow M$ a smooth immersion. Suppose that $\Psi_0(B_R) = \Sigma_0$ is the graph of $u_0 \in C^{\infty}(U)$ where $U \subset P$ is an open subset containing $B_R$ and ${u_0}_{|_{\partial B_R}} = \varphi.$  Then the initial value problem 
 	
 	\begin{equation}
 	\begin{cases}\label{ball-case-eq}
 	& \frac{\partial \Psi}{\partial t}(x,t) = H(\Psi(x,t)), \quad  \mbox{in} \quad  B_R \times (0, T_R) \\
 	& \Psi(x,0) = \Psi_0(x) = \Phi (x, u_0(x)), \quad \mbox{in} \quad   B_R\times\{0\} \\
 	& \Psi(x,t) = \Phi(x, \varphi(x)), \quad \mbox{on} \quad  \partial B_R\times [0, T_R]
 	\end{cases}
 	\end{equation}
 	has a unique smooth graph solution in $B_R \times  [0, T_R]$ with $T_R= \frac{1}{2}\zeta(R). $
 \end{theorem}
 \begin{proof}
    It's enough to prove that there exists $u \in C^\infty(B_R \times (0, T_R)) \cap C(\overline{B_R} \times [0, T_R])$
    solution of the following problem
    \begin{equation} \label{compact-case-problem}
         \begin{cases}
            &\frac{\partial u}{\partial t} = \left (g^{ij}- \frac{u^iu^j}{W^2}\right)u_{i;j} + \left(1 + \frac{1}{\varrho^2 W^2}\right)(\log \varrho)^iu_i,  \quad  \mbox{in} \quad B_R \times (0, T_R) \\
            &u(x,0)  = u_0(x),  \quad \mbox{in} \quad B_R \times \{0\}\\
            &u(x,t) = \varphi(x) \quad \mbox{if} \quad \mbox{on} \quad  \partial B_R \times [0,T_R].
        \end{cases}
    \end{equation}
    Then we have that $\Psi(x,t) = \Phi(x, u(x,t))$ is solution for $(\ref{ball-case-eq}).$
 	\par We have that the problem $(\ref{compact-case-problem})$ with $T_R = \frac{1}{2}\zeta(R)$ is uniformly parabolic, by our a priori gradient estimates. Then there exists $\Lambda >0$ such that the problem
 	\begin{equation} \label{existence-short-time}
 	\begin{cases}
 	&\frac{\partial u}{\partial t} = \left ( g^{ij}- \frac{u^iu^j}{W^2} \right)u_{i;j} + \left(1 + \frac{1}{\varrho^2 W^2}\right)(\log \varrho)^iu_i,  \quad  \mbox{in} \quad \Omega_\Lambda := B_R \times (0, \Lambda) \\
 	&u(x,0)  = u_0(x),  \quad \mbox{in} \quad B_R \times \{0\} \\
    &u(x,t) = \varphi (x) \quad \mbox{on} \quad \partial B_R \times [0,\Lambda]
    \end{cases}
 	\end{equation}
 	has a solution $u^\Lambda $ (see Theorem $8.2$ in \cite{Lieberman}). Moreover $u^\Lambda \in C^\infty(\Omega_\Lambda) \cap C(\overline{\Omega_\Lambda})$ (by  Theorem $8.2$, Theorem $5.14$ in \cite{Lieberman} and linear theory). We note that for $\Lambda >0$ such that the problem $(\ref{existence-short-time})$ has a solution $u^\Lambda,$ our a priori gradient estimate gives us a Holder estimate (by \cite{Lieberman}, Theorem $12.10$) for $u^\Lambda$ which is independent of $\Lambda$, by Corollary \ref*{unif-height-est}. Thus, there exists a solution $u$ for the problem $\eqref{compact-case-problem}$ (see Theorem $8.3$ in \cite{Lieberman}), this solution $u$ is unique by the parabolic comparison principle and $u \in C^\infty(B_R\times(0,T_R))\cap C(\overline{B_R}\times[0,T_R])$ by Schauder estimates.
 	
 \end{proof}

 \subsection{Existence of the flow in non-compact case}

\begin{proof}[Proof of Theorem \ref{main-2}]

From now on, if $R >0$ we denote by  $u^R$ the solution of the $R$-approximate problem that is, the problem 
$\eqref{compact-case-problem}$ in $B_R(o)\times [0, T_R),$  which existence is ensured by Theorem \ref{ball-case}. We also denote
\[
\Psi^R _t(x) = \Phi(x, u^R(x,t)) \quad \mbox{and} \qquad \Sigma_t^R = \Psi_t^R(B_R(o)).
\]
\noindent For a fixed $r_0 > 0,$ we consider $\Lambda_0 > r_0$ the smallest integer belonging to the set
$\left\{\Lambda; \quad \Lambda > r_0, \quad  \zeta(r) < \frac{1}{4}\zeta(\Lambda) \quad \forall \quad  r \leq r_0 \right\}$ and we take 
$$\mathcal{I}_0 = \left\{\Lambda; \quad \Lambda \ge \Lambda_0, \quad  \zeta(r) < \frac{1}{4}\zeta(\Lambda) \quad \forall \quad  r \leq r_0 \right\}.$$
If $\Lambda \in \mathcal{I}_0$ we denote $u^{\Lambda,0} : = u^\Lambda,$  

\begin{equation*}
\begin{split}
    K_{t,\Lambda, \Lambda_0} &= \Sigma_t^{\Lambda} \cap B_{\Lambda_0}(o) \times [0, +\infty) \quad 
    \\
    \mbox{and} \quad K_{t, \Lambda_0, \Lambda_0} &= \Sigma_t^{\Lambda_0} \cap B_{\Lambda_0}(o)) \times [0, + \infty).
\end{split}
\end{equation*}
 We note that $K_{0, \Lambda, \Lambda_0} = K_{0, \Lambda_0,\Lambda_0}$ since the initial condition for $r-$approximate problem is ${\varphi}_{|_{B_r (o)}}$ for all $r>0.$
Since $K_{0, \Lambda, \Lambda_0}$ is compact, there exists hypersurfaces $M_1, M_2$  such that $M_1$ is a translation of the Killing graph of the function $v_{\Lambda_0}, \ M_2$ is a reflection of $M_1$ with respect to the leaf  $P\times\{0\}$ and $K_{0, \Lambda, \Lambda_0}$ is in the strip bounded above and below by $M_2$ and $M_1,$ respectively.

 \noindent We take $T_0= \frac{1}{2}\zeta(\Lambda_0)$ and  $$\ell_0 =  \min \big\{ \ell \in \mathbb{N} ; \, R_0(t) \leq \ell \Lambda_0  \quad \forall \quad t \in [0, T_0] \big\}$$ where $R_0(t) = \mu(t) + \Lambda_0$ is implicitly defined  by $(\ref{def-mu}).$   By using the comparison principle for the mean curvature flow and the Proposition $\ref{estimate-u_+}$ we have
\[
\sup_{B_{r_0}(o) \times [0, T_0]}|u^{\Lambda,0}(x,t)| \leq \sup_{B_{\Lambda_0}(o) \times [0, T_0]}|u^{\Lambda,0}(x,t)| \leq c_0,
\]
for all $\Lambda \in \mathcal{I}_0,$ where the constant $c_0$ depends of $\Lambda_0, \ell_0, \sup_{B_{\Lambda_0}(o)}|u_0|$, $\disp{\varrho_{|_{B_{\Lambda_0}(o)}}}$, $\disp{\xi_{|_{B_{\Lambda_0}(o)}}}$ and  $\zeta(\Lambda_0).$

Moreover, it follows from the Corollary $\ref{unif-grad-est}$ and Corollary $\ref{unif-curv-estimate}$ that for all $\Lambda \in \mathcal{I}_0$ we get
\[
\sup_{B_{r_0}(o) \times [0, T_0]} |\nabla u^{\Lambda, 0}(x,t)| \leq c_1,
\]
and for $m>1$

\[
\sup_{B_{r_0}(o) \times [0, T_0]} |\nabla^m u^{\Lambda, 0}| \leq c_m
\]
where $c_1$ is a constant which depends of $c_0$ and on the geometric data restrict to $B_{\Lambda_0}(o)$ and  $c_m$ is a constant which depends on $c_{m-1}, W^2_{|_{B_{\Lambda_0}(o) \times [0, T_0]}}$ and on the geometric data restricted to $B_{\Lambda_0}(o).$ By using the Arzelà-Ascoli Theorem, we have that there exists a sequence $(\Lambda_{\ell})_\ell$ in $\mathcal{I}_0$ with $\Lambda_\ell \to \infty $ as $\ell \to \infty$ and such that $u^{\Lambda_\ell, 0}$ converges uniformly in $C^{\infty}$ to some $v^0 \in C^{\infty}(B_{r_0} \times [0, T_0])$ which solves $\eqref{compact-case-problem}.$

Let us consider a sequence $\{r_k \}_{k=0}^\infty$ such that $r_0 < r_1 < \cdots $ and $ r_k \to \infty$ as $k \to \infty.$ For each $ k \ge 1$  we consider $\Lambda_k $ the smallest integer belonging to the set
 $\left\{\Lambda; \quad \Lambda > r_k, \quad  \zeta(r) < \frac{1}{4}\zeta(\Lambda) \quad \forall \quad  r \leq r_k \right\}$ and we take 
 $$\mathcal{I}_k = \left\{\Lambda; \quad \Lambda \ge \Lambda_k, \quad  \zeta(r) < \frac{1}{4}\zeta(\Lambda) \quad \forall \quad  r \leq r_k \right\} \qquad \mbox{and} \qquad T_k = \frac{1}{2}\zeta(\Lambda_k).$$ We claim that is possible to get functions $v^{k} \in C^{\infty}(B_{r_k} \times [0, T_k])$ solving $\eqref{compact-case-problem}$ such that $v^{k}$ is the uniform limit of some sequence $\{u^{\Lambda_j, k} \}_{j=1}^\infty$ and ${v_{|_{B_{r_\ell} \times [0,T_l]}}^{k}} =  v^{\ell}$ for all $0 \leq \ell \leq k.$ \\
 
 We will use induction. For $k=0,$ we was done above. The interior estimates imply that we have uniform bounds of $u^{\Lambda}$ and its derivatives on $B_{r_{k + 1}} \times [0,T_{k+1}]$ for all $\Lambda \in \mathcal{I}_{k+1}.$ 
 Then we choose a subsequence of $u^{\Lambda_\ell, k}$ (which will also denote by $u^{\Lambda_\ell, k}$)  such that $\Lambda_\ell \in \mathcal{I}_{k+1}.$ By using the Arzelà-Ascoli Theorem for this subsequence we know that there exist a subsequence $\{u^{\Lambda_{\ell},{k+1}} \}_{\ell}$ of $\{u^{\Lambda_{\ell},k} \}_{\ell=1}^\infty$ such that $u^{\Lambda_\ell,k+1}$ converges uniformly for some $v^{k+1} \in C^{\infty}(B_{r_{k+1}}\times [0,T_{k+1}])$ as $\ell \to \infty.$ Observe that, since $B_{r_k} \times [0,T_k] \subset B_{r_{k+1}}\times [0, T_{k+1}]$ and $\{u^{\Lambda_\ell, k+1} \}_\ell$ is a subsequence of $\{u^{\Lambda_\ell, k} \}_\ell$ we must have  ${v_{|_{B_{r_k} \times [0,T_k]}}^{k+1}} =  v^k.$ For $(x,t) \in P \times[0, \infty),$ we take $k \geq 0$ such that $(x, t) \in B_{r_k} \times [0, T_k]$ and we define $u(x,t) = v^k(x,t)$. It follows from our construction that $u$ is well-defined. If $(x,t) \in \partial_{\infty} P \times [0, \infty),$ we define $u(x,t) = \varphi(x).$ 
 
 We need to show that $u$ is continuous in $(x,t)$ as $x \in \partial_{\infty} P.$
  
  Given $(x_0, t_0) \in \partial_{\infty}P \times [0, \infty) $ and $\Lambda >0,$ there exists an open subset $W \subset \partial_{\infty}P$ such that $x_0 \in W$ and $\varphi(y) < \varphi(x_0) + \frac{\Lambda}{2}$ for all $y \in W.$ Since $P$ is regular at infinity with respect to $\partial_t - Q$ there exists an open subset $U \subset P$ such that $x_0 \in int(\partial_{\infty}U) \subset W$ and $\eta : P \times [0, \infty) \longrightarrow \mathbb{R}$ upper barrier with respect to $(x_0, t_0)$ and $U \times [0, \infty)$ with height $C:= 2 \max_{\bar{P}} |\widetilde{\varphi}|,$ where
  
  $$\widetilde{\varphi}(x) = \left \{\begin{array}{ll}
  u_0(x),& \mbox{if} \quad x \in P \\
  \varphi(x), & \mbox{if} \quad x \in \partial_{\infty}P
  \end{array} \right.$$

  Let $v(x,t) = \eta (x,t) + \varphi(x_0) + \Lambda.$ We want to prove that $u \leq v$ in $U \times [0, \infty).$ For this we will use the sequence $\{u^{\Lambda_{\ell}} = u^{\Lambda_{\ell, 0}} \}_\ell$ where each $u^{\Lambda_{\ell}}$ is solution of 
  \begin{equation}
  	\begin{cases}
  	&(\partial_t - Q)[u] = 0 \quad \mbox{in} \quad B_{\Lambda_\ell} \times [0, T_{\Lambda_\ell}] \\
  	&u(x.0) = u_0(x), \quad  x \in  B_{\Lambda_\ell} \\
  	&u(x,t) = u_0(x), \quad x \in  \partial B_{\Lambda_\ell}  \quad \mbox{and} \quad t \in  [0, T_{\Lambda_\ell}].
  	\end{cases}
  \end{equation}
  Since  $\widetilde{\varphi}(x)$ is continuous, there exists $\ell_0 > > 1$ such that
 $$u_0(x) < \varphi(x_0) + \frac{\Lambda}{2} \qquad \forall \quad x \in \partial B_{\Lambda_\ell} \cap U, \qquad \Lambda_{\ell} \geq \ell_0. $$
 We observe that $u^{\Lambda_{\ell}} \leq v \quad \mbox{in} \quad B_{\Lambda_{\ell}} \times [0, T_{\Lambda_{\ell}}] \cap (U \times [0, \infty))$ for $\Lambda_{\ell} \geq \ell_0.$ In fact, if $(x,t) \in \partial B_{\Lambda_{\ell}} \times [0, T_{\Lambda_{\ell}}]\cap U \times [0, \infty),$ 
 $$u^{\Lambda_{\ell}}(x,t) = u_0(x) < \varphi (x_0) + \frac{\Lambda}{2} \leq v(x,t)$$ due to the choice of $\ell_0.$
If
 $(x,t) \in B_{\Lambda_{\ell}} \times [0, T_{\Lambda_{\ell}}]\cap (\partial U \times [0, \infty))$ we have 
 $$ u^{\Lambda_{\ell}}(x,t) \leq \max_{\partial  B_{\Lambda_{\ell}}} \widetilde{\varphi} \leq 2 \max_{P\cup \partial_{\infty}P}|\widetilde{\varphi}| + \varphi (x_0) \leq \varphi(x_0) +   \eta(x,t) \leq v(x,t).$$
 Since we have that 
    \begin{equation}
    \begin{split}
        \partial (B_{\Lambda_{\ell}} \times [0, T_{\Lambda_{\ell}}] \cap U \times [0, \infty)) &= \overline{\partial B_{\Lambda_{\ell}}\times [0, T_{\Lambda_{\ell}}] \cap U \times [0, \infty)} 
        \\
        &\cup \overline{\partial U \times [0, \infty) \cap B_{\Lambda_{\ell}}\times [0, T_{\Lambda_{\ell}}]},
    \end{split}
    \end{equation}
    it follows from Comparison Principle that $$u^{\Lambda_{\ell}} \leq v \quad \mbox{in} \quad B_{\Lambda_{\ell}} \times [0, T_{\Lambda_{\ell}}] \cap U \times [0, \infty) \quad \forall \quad \Lambda_{\ell} \geq \ell_0.$$ 
 Hence $v^k \leq v \quad \mbox{in} \quad B_{R_k} \times [0, T_{R_k}] \cap U \times [0, \infty) \quad \forall \quad k$  consequently $u \leq v$ in $U \times [0, \infty).$\\
 
 If we define $\widetilde{v}(x,t) = \varphi(x_0) - \Lambda - \eta (x,t),$ we can show that $u \geq \widetilde{v}$ in  $U \times [0, \infty).$ So $$|u(x,t) - \varphi(x_0, t_0)| < \Lambda + \eta (x,t)$$ for all $(x,t) \in U \times [0, \infty).$ Therefore 
 $$\limsup_{(x,t) \to (x_0,t_0)} |u(x,t)- \varphi (x_0, t_0)| \leq \Lambda$$
which implies that $u \in C^{\infty}(P \times [0, \infty)) \cap C^0( \bar{P} \times [0, \infty)). $
\end{proof}

\printbibliography

\end{document}